\documentclass[a4paper]{amsart}
\date{August 6, 2019}
\usepackage[dvipdfmx]{graphicx}
\usepackage{latexsym}
\usepackage{amssymb}
\usepackage{amsmath}
\usepackage{amscd}
\usepackage{amsthm}
\usepackage{amsfonts}
\usepackage{mathrsfs}
\usepackage{enumerate}
\usepackage{enumitem}
\usepackage{bm}
\usepackage[usenames]{color}
\usepackage[active]{srcltx}
\title{Isometric deformations of mixed type surfaces 
in Lorentz-Minkowski space}
\author[A.~Honda]{Atsufumi Honda}
\address{%
   Department of Applied Mathematics, 
   Faculty of Engineering, Yokohama National University, 
   79-5 Tokiwadai, Hodogaya, Yokohama 240-8501, Japan
}
\email{honda-atsufumi-kp@ynu.ac.jp}
\thanks{This work was supported by 
JSPS KAKENHI Grant Number 19K14526.}
\usepackage{amsthm}
\theoremstyle{plain}
 \newtheorem{theorem}{Theorem}[section]
 \newtheorem{introtheorem}{Theorem}
 \newtheorem{introcorollary}[introtheorem]{Corollary}

 \newtheorem{proposition}[theorem]{Proposition}
 \newtheorem{fact}[theorem]{Fact}
 \newtheorem*{fact*}{Fact}
 \newtheorem{lemma}[theorem]{Lemma}
 \newtheorem{corollary}[theorem]{Corollary}
 \theoremstyle{remark}
 \newtheorem{definition}[theorem]{Definition}
 
 \newtheorem*{acknowledgements}{Acknowledgements}
 \newtheorem{example}[theorem]{Example}
\numberwithin{equation}{section}

 \makeatletter
    
    \@addtoreset{equation}{section}
  \makeatother

\newcommand{\Z}{\boldsymbol{Z}}

\newcommand{\R}{\boldsymbol{R}}

\newcommand{\SO}{\operatorname{SO}}
\renewcommand{\O}{\operatorname{O}}
\renewcommand{\L}{\boldsymbol{L}}
\newcommand{\ep}{\epsilon}
\newcommand{\vect}[1]{\boldsymbol{#1}}
\newcommand{\inner}[2]{\left\langle{#1},{#2}\right\rangle}

\newcommand{\calC}{\mathcal{C}}
\newcommand{\calF}{\mathcal{F}}
\newcommand{\calG}{\mathcal{G}}
\newcommand{\calU}{\mathcal{U}}
\newcommand{\calV}{\mathcal{V}}

\begin{document}
\begin{abstract}
A connected regular surface in Lorentz-Minkowski 3-space
is called a {\it mixed type surface\/} if 
the spacelike, timelike and lightlike 
point sets are all non-empty.
Lightlike points on mixed type surfaces 
may be regarded as singular points of 
the induced metrics.
In this paper, 
we introduce the \emph{L-Gauss map}
around non-degenerate lightlike points,
and show 
the fundamental theorem of surface theory 
for mixed type surfaces at non-degenerate lightlike points.
As an application,
we prove that a real analytic mixed type surface 
admits non-trivial isometric deformations
around generic lightlike points.
\end{abstract}
\maketitle


\section{Introduction}
Let us denote by $\L^3$ the Lorentz-Minkowski $3$-space
of signature $(++-)$.
Consider an embedded surface $S$ in $\L^3$.
The induced metric of $S$ might be either 
positive definite, indefinite, or degenerate.
According to such properties,
a point on the surface $S$ is said to be 
spacelike, timelike, or lightlike.
If the spacelike, timelike and lightlike point sets
are all non-empty,
the surface $S$
is said to be a \emph{mixed type surface}.

\begin{figure}[htb]
\begin{center}
 \begin{tabular}{{c@{\hspace{20mm}}c}}
  \resizebox{3cm}{!}{\includegraphics{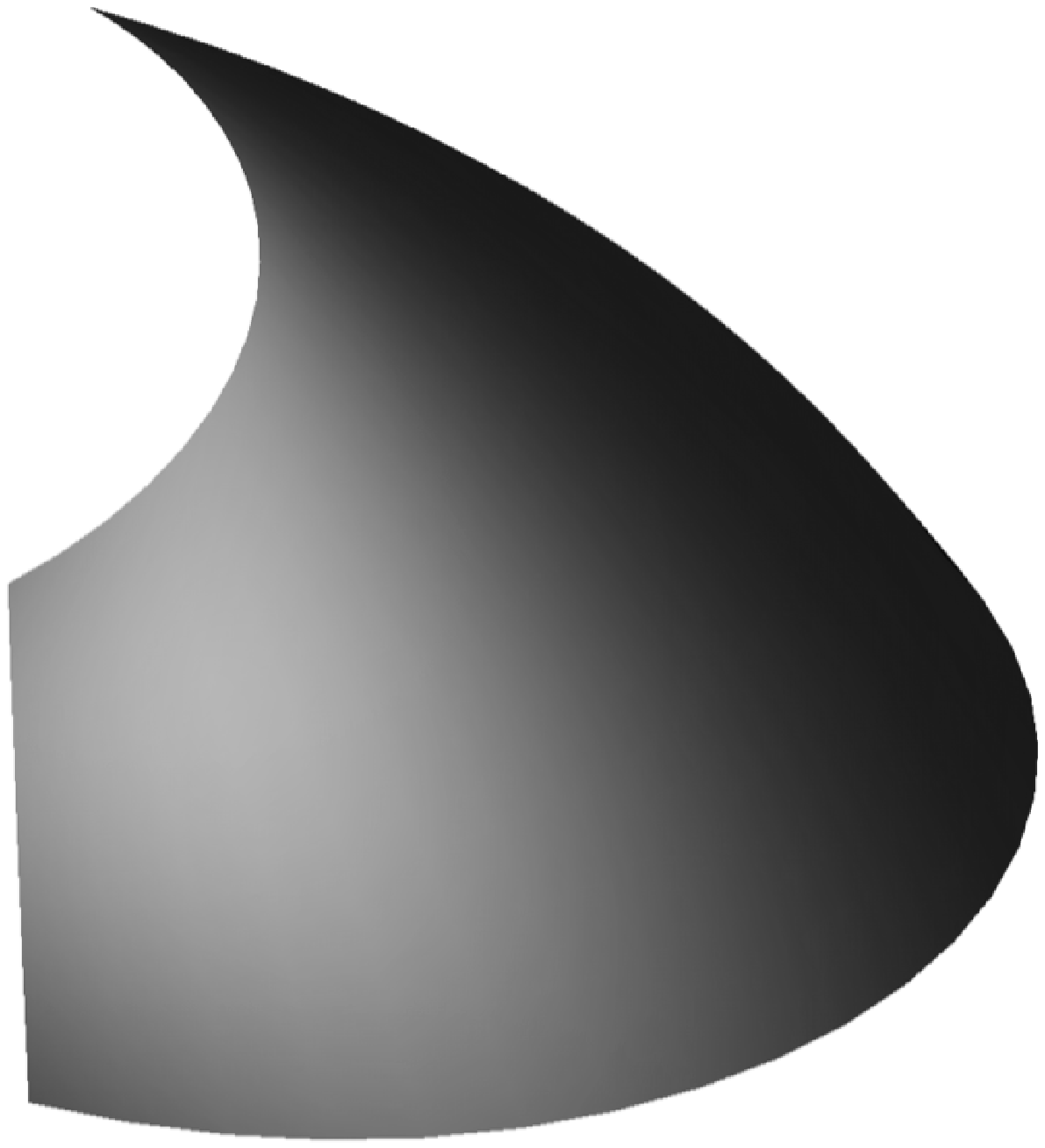}} &
  \resizebox{3cm}{!}{\includegraphics{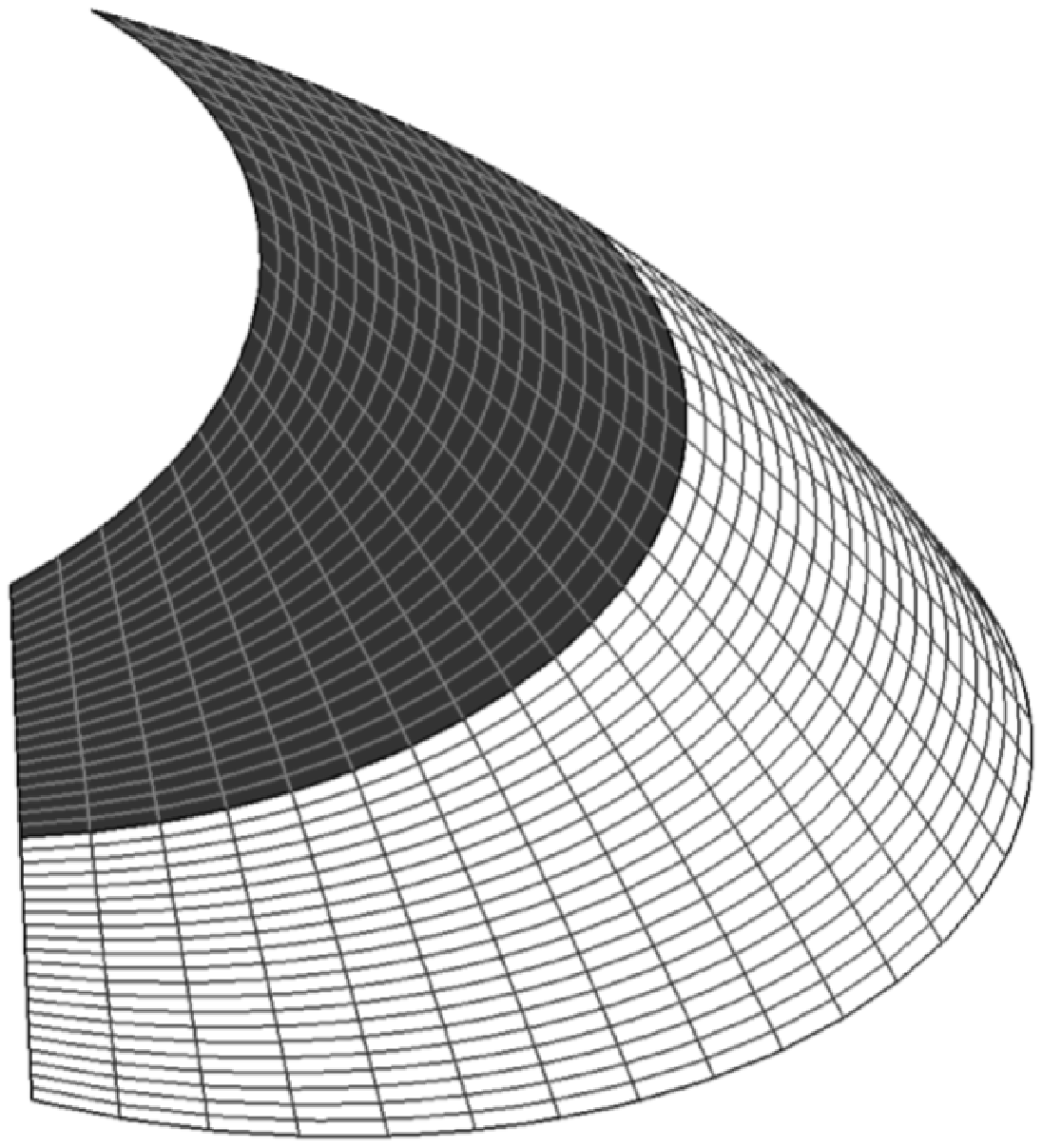}}
 \end{tabular}
 \caption{The images of a mixed type surface 
 in the Lorentz-Minkowski $3$-space $\L^3$.
 In the right figure,
 the dark (resp.\ light) colored region 
 shows the spacelike (resp.\ timelike) point set.
 The boundary curve is the lightlike set image,
 cf.\ $f_1$ in Example \ref{ex:four}.
 }
 \label{fig:f1-mixed}
\end{center}
\end{figure}

It is known that every closed surface, 
i.e.\ a compact surface without boundary,
in $\L^3$ 
must be of mixed type
\cite{Tari2013}.
The first fundamental form of a mixed type surface
is a type-changing metric.
Type-changing metrics are of interest both in 
mathematics and theoretical physics,
see references in \cite{Pavlova-Remizov, HST}.
Recently, mixed type surfaces of zero mean curvature, 
especially the so-called Bernstein-type property, 
have been investigated intensively
(see \cite{FKKRUY_qj, AUY1, AUY2, AHUY} and the references therein).
As a broader class, mixed type surfaces of bounded mean curvature 
are also being studied
\cite{HKKUY, UY_geloma, UY_2018}.

Since lightlike points may be
regarded as {\it singular points} of the induced metric,
methods to study singularities
have also been applied to investigate 
lightlike points of mixed type surfaces,
cf.\ \cite{Tari2012, IzumiyaTari2010, IzumiyaTari2013, Rem-Tari}.
In \cite{HST},
the author with K.~Saji and K.~Teramoto
investigated lightlike points of mixed type surfaces 
in the way similar to the case of {\it wave fronts} \cite{SUY1}.
To study the behavior of the Gaussian curvature 
at lightlike points, in \cite{HST},
the invariants named
the {\it lightlike singular curvature} $\kappa_L$
and the {\it lightlike normal curvature} $\kappa_N$
were introduced (cf.\ Definition \ref{def:invariants}).

Here, we briefly review the 
intrinsic and extrinsic invariants.
Let $S$ be a mixed type surface given by 
an embedding $f:\Sigma \to \L^3$  
of a connected smooth $2$-manifold $\Sigma$.
We denote by $ds^2$ the first fundamental form of $f$,
and by $LD\subset \Sigma$ 
the set of lightlike points.
Then, an {\it invariant} is a function 
$I : \Sigma\to \R$, or $I : LD\to \R$,
that is, $I$ does not depend on the choice of 
coordinate system of the source.
Moreover, 
\begin{itemize}
\item
an invariant $I$ is called {\it intrinsic}
if it depends only on $ds^2$.
Namely, $I$ can be locally represented by 
a function of 
$E$, $F$, $G$ and their derivatives,
where $ds^2=E\,du^2+2F\,du\,dv+G\,dv^2$,
and $(u,v)$ is a coordinate defined in terms of 
the first fundamental form $ds^2$.
\item
an invariant $I$ is called {\it extrinsic}
if there exists a mixed type surface $\tilde{f}$ 
such that the first fundamental form of $\tilde{f}$
is the same as for $f$,
but $I$ does not coincide on $\tilde{f}$ and $f$.
\end{itemize}
To determine the intrinsicity and extrinsicity of invariants
is one fundamental problem.
In \cite{HST}, it is proved that 
the lightlike singular curvature $\kappa_L$
is an intrinsic invariant.
Moreover, it is also proved that
the lightlike normal curvature $\kappa_N$
is also intrinsic if $\kappa_L$ 
is identically zero along $LD$
(\cite[Corollary C]{HST}, Fact \ref{fact:N-int}).
However, it is not known whether 
$\kappa_N$
is intrinsic or extrinsic, in general.

In this paper, we prove that 
every real analytic {\it generic} mixed type surface admits 
non-trivial local isometric deformations,
which yields the extrinsicity of 
the lightlike normal curvature $\kappa_N$.

\subsection{Statement of results}
To be more precise,
let $f:\Sigma\to \L^3$ be a mixed type surface.
Denote by $ds^2$ the first fundamental form of $f$.
Lightlike points can be characterized as the points
where $ds^2$ degenerates.
On a coordinate neighborhood $(U;u,v)$,
the first fundamental form $ds^2$ is expressed as
$ds^2=E\,du^2+2F\,du\,dv+G\,dv^2$,
and 
then the lightlike set $LD$ is locally written as
the zero set of $\lambda:=EG-F^2$.
A lightlike point $p$ satisfying
$d\lambda(p)\ne0$ is said to be {\it non-degenerate}.
By the implicit function theorem, 
$LD$ can be locally parametrized by a regular curve
$c(t)$ $(|t|<\delta)$, called the {\it characteristic curve\/},
passing through $p=c(0)$, where $\delta>0$.
Then, $\hat{c}(t):=f \circ c(t)$ is a regular curve in $\L^3$.
If $\hat{c}'(0)$ is spacelike (resp.\ lightlike),
then $p\in LD$ is said to be a {\it lightlike point of the first kind}
(resp {\it the second kind\/}).
If $p$ is a lightlike point of the first kind,
$\hat{c}(t)$ is 
a spacelike regular curve in $\L^3$ near $p$.

A regular curve $c(t)$ $(|t|<\delta)$
passing through $p=c(0)$
is called non-null at $p$
if $c'(0)$ is not a null vector,
where $\delta>0$.
Then, we define the genericity of lightlike points
as follows:

\begin{definition}
\label{def:generic-surface}
Let $f:\Sigma\to \L^3$ be a mixed type surface.
A non-degenerate lightlike point $p\in LD$ is called {\it generic}
if $p$ satisfies one of the following conditions:
\begin{itemize}
\item[(I)] 
If $p$ is of the first kind, 
the lightlike singular curvature
$\kappa_L$ does not vanish at $p$.
\item[(II)]
If $p$ is of the second kind, 
there exists a non-null curve
$c(t)$ $(|t|< \delta)$
at $p=c(0)$
such that the geodesic curvature function 
$\kappa_g(t)$ along $c(t)$ defined for $t\ne0$
is unbounded at $t=0$.
\end{itemize}
If all lightlike points are generic,
the mixed type surface is said to be {\it generic}.
\end{definition}

The condition (II) of the genericity can be characterized
by an invariant called 
the {\it limiting geodesic curvature} $\mu_c$,
cf.\ Proposition \ref{prop:limiting-geod-2nd} and
Corollary \ref{prop:limiting-geod-2nd}.
The main result of this paper is as follows:

\begin{introtheorem}\label{thm:main}
Let $f: \Sigma \rightarrow \L^3$ be 
a real analytic generic mixed type surface
with the first fundamental form $ds^2$,
and let $p\in \Sigma$ be a lightlike point.
We also let $c(t)$ $(|t|<\delta)$ be either
\begin{itemize}
\item a characteristic curve passing through $p=c(0)$, 
if $p$ is of the first kind, or
\item a regular curve which is non-null at $p=c(0)$
such that the geodesic curvature function 
is unbounded at $t=0$, if $p$ is of the second kind,
\end{itemize}
where $\delta>0$.
Take a real analytic spacelike curve $\gamma : I \to \L^3$
with non-zero curvature 
passing through $\gamma(0)=\vect{0}$
and set the image $\Gamma:=\gamma(I)$ of $\gamma$,
where $I$ is an open interval including the origin $0$.
Set $Z_\gamma$ as
$$
  Z_\gamma :=  
  \begin{cases} 
  \{1,2,3,4\} & (\text{if $\gamma$ is a Frenet curve}),\\ 
  \{1,2\} & (\text{if $\gamma$ is a non-Frenet curve}). 
  \end{cases}
$$
Then,
there exist a neighborhood $U$ of $p$
and real analytic mixed type surfaces 
$f_i : U \rightarrow \L^3$ $(i\in Z_\gamma)$
such that, for each $i\in Z_\gamma$,
\begin{itemize}
\item[$(1)$]
the first fundamental form of $f_i$ coincides with $ds^2$,
\item[$(2)$]
$f_i(p)=\vect{0}$, and
the image of $f_i\circ c(t)$ is included in $\Gamma$.
\end{itemize}
Moreover, there are no such surfaces 
other than $f_i$ $(i\in Z_\gamma)$.
More precisely,
if $\tilde{f} : U \rightarrow \L^3$
is a real analytic generic mixed type surface 
which satisfies the conditions $(1)$ and $(2)$,
then there exists an open neighborhood $O$ of $p$ 
such that the image $\tilde{f}(O)$ is a subset 
of $f_i(U)$ for some $i\in Z_\gamma$.
\end{introtheorem}

For the definition of Frenet curves,
see Definition \ref{def:non-Frenet}.
Example \ref{ex:four} gives
such the surfaces $f_i$ $(i=1,2,3,4)$
in the case of $\#Z_\gamma=4$,
namely, the case that $\gamma$ is a Frenet curve,
cf.\ Figures \ref{fig:four} and \ref{fig:four-sep}.
See Example \ref{ex:two}
for the case that $\gamma$ is a non-Frenet curve.

\begin{figure}[htb]
\begin{center}
 \begin{tabular}{{c@{\hspace{10mm}}c}}
  \resizebox{5.5cm}{!}{\includegraphics{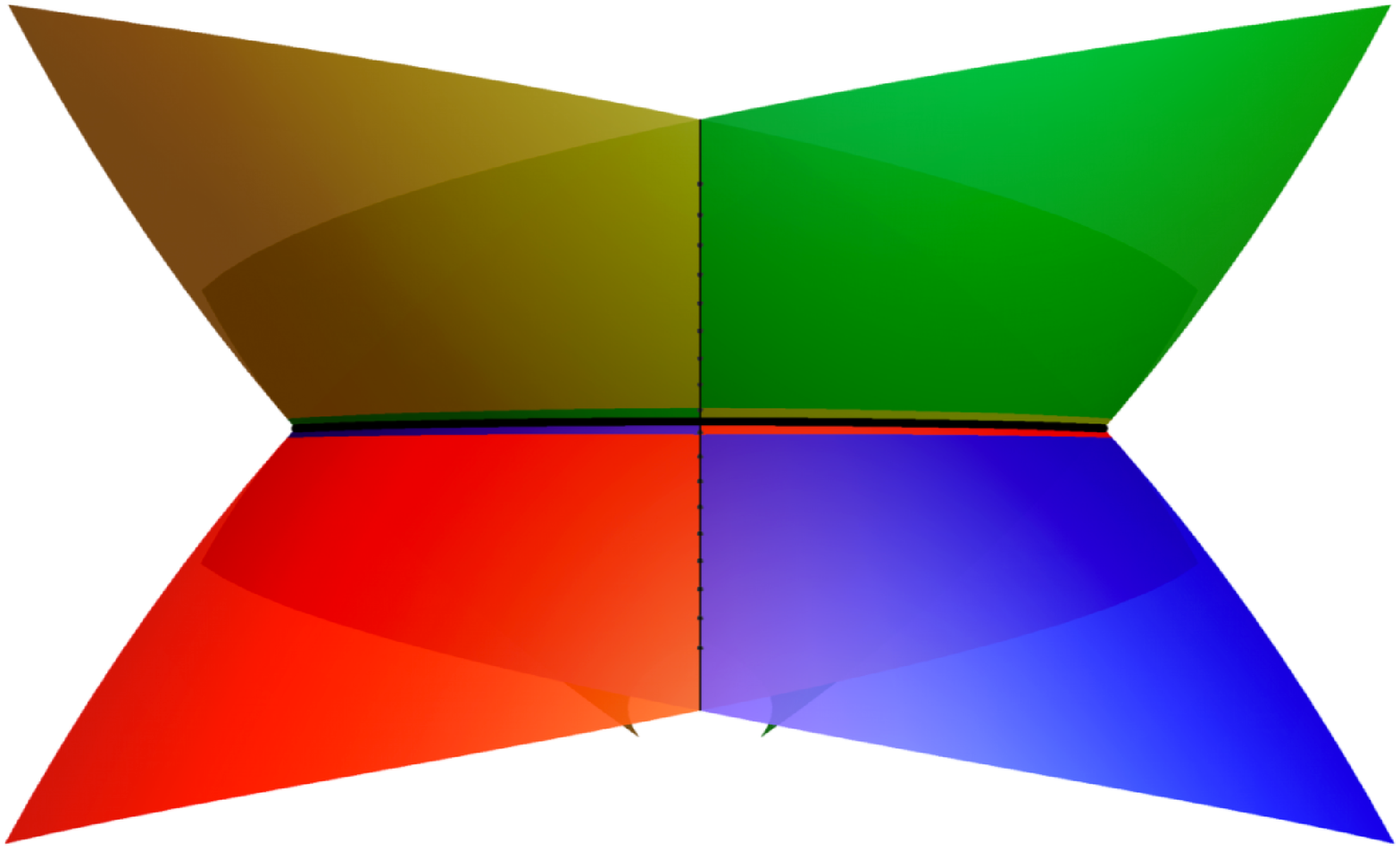}} &
  \resizebox{5.5cm}{!}{\includegraphics{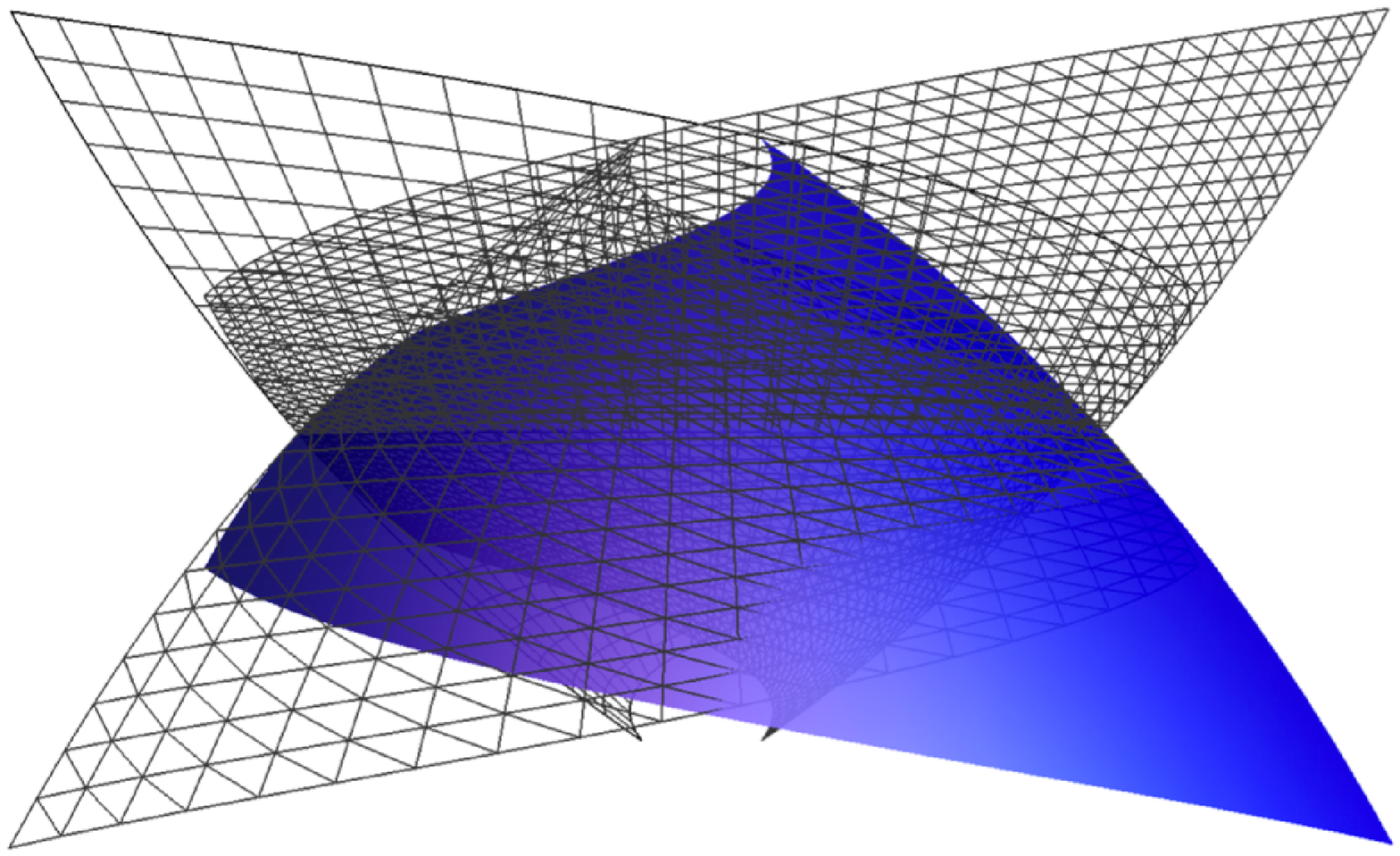}}
 \end{tabular}
 \caption{Image of the four mixed type surfaces 
 $f_1$, $f_2$, $f_3$ and $f_4$ (left) 
 and its transparent (right), see Example \ref{ex:four}.
 The surfaces $f_i$ $(i=1,2,3,4)$ have the same first fundamental form,
 and their lightlike set images are subsets of
 a common spacelike Frenet curve.
 }
 \label{fig:four}
\end{center}
\end{figure}

As pointed out in \cite{HST},
lightlike points of the first kind of mixed type surfaces
are similar to the cuspidal edge singularity of wave fronts
\cite{KRSUY}.
The invariants $\kappa_L$ and $\kappa_N$
of lightlike points of the first kind
have several properties 
similar to the invariants of cuspidal edges,
called the {\it singular curvature} $\kappa_s$
and the {\it limiting normal curvature} $\kappa_\nu$
introduced in \cite{SUY1} (cf.\ \cite{MS, MSUY}).
However, sometimes $\kappa_L$ and $\kappa_N$
behave quite differently than $\kappa_s$ and $\kappa_\nu$,
cf.\ \cite[Theorem A]{HST}.
Isometric realizations of 
a class of positive semidefinite metrics
(called the {\it Kossowski metrics} \cite{Kossowski, HHNSUY}),
and isometric deformations of wave fronts 
at non-degenerate singular points,
including cuspidal edges,
are discussed in \cite{Kossowski, NUY, HNUY, HNSUY1, HNSUY2, Honda-Saji}.
In particular, we remark that,
for a real analytic generic cuspidal edge $f$
with the induced metric $ds^2$
in the Euclidean $3$-space,
there exist four cuspidal edges $f_1$, $f_2$, $f_3$ and $f_4$
such that the each induced metric of $f_i$ $(i=1,2,3,4)$
coincides with $ds^2$
and share a common prescribed singular set images
\cite{NUY, HNUY, HNSUY1}.
Hence, Theorem \ref{thm:main} 
may be considered as a similar result,
in the case of Frenet curves.
However, in the case of non-Frenet curves,
Theorem \ref{thm:main} provides 
a phenomena different from cuspidal edges.

As a corollary of Theorem \ref{thm:main},
we have the following:

\begin{introcorollary}\label{cor:deformation}
Every real analytic generic mixed type surface admits 
non-trivial local isometric deformations
at their lightlike points.
\end{introcorollary}

Using this isometric deformation,
we obtain the desired extrinsicity 
of the lightlike normal curvature $\kappa_N$:

\begin{introcorollary}\label{cor:ext-kappa-N}
The lightlike normal curvature $\kappa_N$
is an extrinsic invariant 
for real analytic generic mixed type surfaces.
\end{introcorollary}

\subsection{Organization of this paper}
To prove Theorem \ref{thm:main},
we prepare a theorem so-called 
`fundamental theorem of surface theory' 
for mixed type surfaces 
(Theorem \ref{thm:fundamental}).
Since the unit normal vector field
along the surfaces cannot be bounded
at their lightlike points,
we use an alternative transversal vector field
called the {\it L-Gauss map} 
(Lemma \ref{lem:L-Gauss}).
More detailed summary
of this paper is as follows.

First, in Section \ref{sec:prelim}, 
we review the fundamental properties 
of lightlike points of mixed type surfaces.
The definitions of the invariants of
lightlike points of the first kind,
such as the lightlike singular curvature $\kappa_L$
and the lightlike normal curvature $\kappa_N$,
are also given (Definition \ref{def:invariants}).

In Section \ref{sec:metric},
we shall discuss a class of type-changing metric
called {\it admissible mixed type metrics}
(Definition \ref{def:admissible-metric}),
which is modeled on the first fundamental form
of mixed type surfaces.
In particular,
the existence of a special orthogonal coordinate system
called an {\it L-coordinate system}
is proved
(Proposition \ref{prop:L-orthogonal}).
We also introduce 
invariants called
the intrinsic lightlike singular curvature 
$\tilde{\kappa}_L$
at type I semidefinite points 
(Definition \ref{def:kappa_int}),
and the limiting geodesic curvature 
$\mu_c$
at type II semidefinite points 
(cf.\ \eqref{eq:limiting-g-curvature}),
which are related to the definition
of the genericity of metrics
(cf.\ Definitions 
\ref{def:generic-typeI}, 
\ref{def:generic-typeII}).

Section \ref{sec:L-Gauss} is the core of this paper.
Here, we prove the existence of a transversal vector field
$\psi$ along a mixed type surface in $\L^3$,
which we call the {\it L-Gauss map} 
(Lemma \ref{lem:L-Gauss}),
where we use L-coordinate systems
and the division lemma 
(cf.\ \cite[Appendix A]{UY_geloma}).
Then,
we consider a frame associated with 
the L-Gauss map.
The compatibility condition of the frame
(Lemma \ref{lem:integrability})
yields the fundamental theorem of mixed type surfaces
(Theorem \ref{thm:fundamental}).

In Section \ref{sec:curve},
we calculate the invariants of 
the spacelike curve $\hat{c}(t):=f\circ c(t)$
given by a non-null curve $c(t)$ on $\Sigma$
for the proof of Theorem \ref{thm:main}.
We remark that the curvature vector 
$\vect{\kappa}(t)$ of a spacelike curve
may be of mixed type, cf.\ \cite{Honda-Lk}.
The definition of spacelike Frenet/non-Frenet curves,
and their invariants,
such as the curvature function, the torsion function
and the pseudo-torsion function,
are also reviewed here.

Finally, in Section \ref{sec:proofs},
we prove the isometric realization theorem
of real analytic generic mixed type metrics
(Theorem \ref{thm:realization}, Corollary \ref{cor:realization}).
Such the realization theorem may be regarded 
as an analogue of the well-known Janet--Cartan theorem 
\cite{Janet, Cartan}.
See \cite{Bergner} for isometric embeddings of 
Riemannian or Lorentzian metrics
as first fundamental forms for regular surfaces in $\L^3$.
As a corollary,
we prove Theorem \ref{thm:main},
and Corollaries \ref{cor:deformation} and \ref{cor:ext-kappa-N}.
We also prove the extrinsicity 
of the lightlike geodesic torsion $\kappa_G$
(Corollary \ref{cor:ext-kappa-G}).

\section{Preliminaries}
\label{sec:prelim}

We denote by $\L^3$ the Lorentz-Minkowski $3$-space
with the standard Lorentz metric
$$
  \inner{\vect{x}}{\vect{x}} 
  = \vect{x}^T S \vect{x}
  =  (x_1)^2 + (x_2)^2 - (x_3)^2
  \qquad
  \left(\vect{x}=\sum_{i=1}^3 x_i \vect{e}_i\in \L^3\right),
$$
where 
\begin{equation}\label{eq:S}
  S:=
     \begin{pmatrix} 
     1 & 0 & 0\\
     0 & 1 & 0\\
     0 & 0 & -1
     \end{pmatrix},\quad
  \vect{e}_1:=\begin{pmatrix}1\\0\\0\end{pmatrix},\quad
  \vect{e}_2:=\begin{pmatrix}0\\1\\0\end{pmatrix},\quad
  \vect{e}_3:=\begin{pmatrix}0\\0\\1\end{pmatrix},
\end{equation}
and $\vect{x}^T$ stands for the transpose of 
the column vector $\vect{x}$.
A vector $\vect{x}\in \L^3$ $(\vect{x} \ne \vect{0})$
is said to be {\it spacelike\/}
if $\inner{\vect{x}}{\vect{x}}>0$.
Similarly, 
if $\inner{\vect{x}}{\vect{x}}<0$
(resp.\ $\inner{\vect{x}}{\vect{x}}=0$),
then $\vect{x}$ is {\it timelike\/} 
(resp.\ {\it lightlike\/}).
For $\vect{x}\in \L^3$, we set
$|\vect{x}|:= \sqrt{|\inner{\vect{x}}{\vect{x}}|}$.
For vectors $\vect{v}, \vect{w} \in \L^3$,
the vector product $\vect{v}\times \vect{w}$
is given by
\begin{equation}\label{eq:CROSS}
  \vect{v}\times \vect{w}:=S \vect{v}\times_E \vect{w},
\end{equation}
where $\times_E$ means 
the standard cross product of 
the Euclidean $3$-space $\R^3$.
Then, 
it holds that
\begin{align}
\label{eq:scalar-triplet}
\det(\vect{u},\vect{v},\vect{w})&=\inner{\vect{u}}{\vect{v}\times \vect{w}},\\
\label{eq:vector-triplet}
\vect{u}\times (\vect{v}\times \vect{w})
&=\inner{\vect{u}}{\vect{v}}\vect{w}-\inner{\vect{u}}{\vect{w}}\vect{v},\\
\label{eq:area-formula}
\inner{ \vect{v}\times \vect{w} }{ \vect{v}\times \vect{w} }
  &= - \inner{\vect{v}}{\vect{v}} \inner{\vect{w}}{\vect{w}}
     + \inner{\vect{v}}{\vect{w}}^2
\end{align}
for $\vect{u},\vect{v},\vect{w}\in \L^3$.
In particular, 
$\vect{v}\times \vect{w}$ is orthogonal to 
$\vect{v}$ and $\vect{w}$.
To calculate the vector product of a lightlike vector, 
the following formula is useful.

\begin{fact}[{cf.\ \cite[Lemma 4.3]{HST}}]
\label{fact:gaiseki}
Let $\vect{v}\in \L^3$ be a spacelike vector.
Take a lightlike vector $\vect{w}\in \L^3$
such that $\inner{\vect{v}}{\vect{w}}=0$.
Then, either
$\vect{v} \times \vect{w} = |\vect{v}| \, \vect{w}$
or $\vect{v} \times \vect{w} = -|\vect{v}| \, \vect{w}$
holds.
\end{fact}

Hence, we use the following terminology.

\begin{definition}\label{def:pn-orient}
Let $\{\vect{v},\vect{w}\}$ be a pair of vectors in $\L^3$
such that $\vect{v}$ is spacelike,
and 
$\vect{w}$ is lightlike and orthogonal to $\vect{v}$.
Then, $\{\vect{v},\vect{w}\}$ is called {\it p-oriented\/}
(resp.\ {\it n-oriented\/})
if $\vect{v} \times \vect{w} = |\vect{v}| \, \vect{w}$ 
(resp.\ $\vect{v} \times \vect{w} = -|\vect{v}| \, \vect{w}$)
holds.
\end{definition}

The isometry group ${\rm Isom}(\L^3)$
of $\L^3$
is described as the semidirect product
$
  {\rm Isom}(\L^3)=\O(1,2) \ltimes \L^3,
$
where $\O(1,2)$ is the Lorentz group
which consists of $3\times 3$ matrices $A$
such that $A^T S A = S$,
where 
$S$ is the matrix given by \eqref{eq:S}.
We call 
\begin{align*}
  \SO(1,2)&:=\left\{ A \in \O(1,2)\,;\, \det A =1 \right\},\\
  \SO^+(1,2)&:=\left\{ A=(a_{ij}) \in \SO(1,2)\,;\, a_{33} >0 \right\}
\end{align*}
the special Lorentz group and the restricted Lorentz group,
respectively.
We remark that $\SO^+(1,2)$ is the identity component of $\O(1,2)$,
and the subgroup $\SO(1,2)\ltimes \L^3$ of ${\rm Isom}(\L^3)$
is the orientation preserving isometry group of $\L^3$.

A basis $\{\vect{v}_1,\vect{v}_2,\vect{v}_3\}$ of $\L^3$ is orthonormal
(i.e.\ $\inner{\vect{v}_i}{\vect{v}_j}=\sigma_i\delta_{ij}$,
where $\sigma_1=\sigma_2=1$, and $\sigma_3=-1$)
if and only if 
the square matrix $(\vect{v}_1,\vect{v}_2,\vect{v}_3)$
is an element of $\O(1,2)$.
For each orthonormal basis 
$\{\vect{v}_i\}_{i=1,2,3}$,
there exists $T\in \SO(1,2)$
such that
$$
   T(\vect{v}_1,\vect{v}_2,\vect{v}_3)
  =
  \begin{cases}
  (\vect{e}_1,\vect{e}_2,\vect{e}_3)
     & (\text{if $\det(\vect{v}_1,\vect{v}_2,\vect{v}_3)=1$}),\\
  (\vect{e}_1,\vect{e}_2,-\vect{e}_3)
     & (\text{if $\det(\vect{v}_1,\vect{v}_2,\vect{v}_3)=-1$}),
  \end{cases}
$$
where $\vect{e}_i$ $(i=1,2,3)$
are given in \eqref{eq:S}.
A basis $\{\vect{w}_1,\vect{w}_2,\vect{w}_3\}$ of $\L^3$ is 
said to be a {\it null basis} if
$$
  \inner{\vect{w}_1}{\vect{w}_1}
  =\inner{\vect{w}_2}{\vect{w}_3}
  =1,\qquad
  \inner{\vect{w}_1}{\vect{w}_i}
  =\inner{\vect{w}_i}{\vect{w}_i}
  =0
$$
for $i=2,3$. 
For each null basis $\{\vect{w}_1,\vect{w}_2,\vect{w}_3\}$,
there exists $T\in \SO(1,2)$
such that
\begin{equation}\label{eq:null-basis}
  T(\vect{w}_1,\vect{w}_2,\vect{w}_3)
  = \begin{cases}
  \left(\vect{e}_1, 
      \frac1{\sqrt{2}} \vect{e}_-, \frac1{\sqrt{2}} \vect{e}_+ \right)
     & (\text{if $\det(\vect{w}_1,\vect{w}_2,\vect{w}_3)=1$}),\\
  \left(\vect{e}_1,
      \frac{-1}{\sqrt{2}} \vect{e}_+, \frac{-1}{\sqrt{2}} \vect{e}_-\right)
   & (\text{if $\det(\vect{w}_1,\vect{w}_2,\vect{w}_3)=-1$}),
  \end{cases}
\end{equation}
where $\vect{e}_+:=\vect{e}_2+\vect{e}_3$,
$\vect{e}_-:=\vect{e}_2-\vect{e}_3$.

\subsection{Lightlike points of mixed type surfaces}

In this paper, a {\it surface} in $\L^3$
is defined to be an embedding 
$$
  f:\Sigma\longrightarrow \L^3
$$
of a connected differentiable $2$-manifold $\Sigma$.
A point $p\in \Sigma$ is said to be a \emph{lightlike}
(resp.\ \emph{spacelike}, \emph{timelike}) \emph{point}
if the image $V_p:=df_p(T_p\Sigma)$
of the tangent space $T_p\Sigma$
is a lightlike (resp.\ spacelike, timelike) $2$-subspace of $\L^3$.
We denote by $LD$ (resp.\ $\Sigma_+$, $\Sigma_-$)
the set of lightlike (resp.\ spacelike, timelike) points.
We also call $LD$ the {\it lightlike set} of $f$.
If both the spacelike sets $\Sigma_+$ and 
the timelike sets $\Sigma_-$ are non-empty,
the surface is called a {\it mixed type surface\/}.

Denote by $ds^2$ the {\it first fundamental form}
(or the {\it induced metric\/}) of $f$,
namely $ds^2$ is the smooth metric on $\Sigma$
defined by $ds^2:=f^*\inner{~}{~}$.
Then, $p\in \Sigma$ is a lightlike point
if and only if $(ds^2)_p$ is degenerate
as a symmetric bilinear form on $T_p\Sigma$.
Similarly, 
$p\in \Sigma$ is a spacelike (resp.\ timelike) point
if and only if $(ds^2)_p$ is positive definite (resp.\ indefinite).
Take a local coordinate neighborhood $(U;u,v)$ 
of $\Sigma$.
Set 
$$
f_u:=df(\partial_u),\quad 
f_v:=df(\partial_u),\quad
\text{where} \quad
\partial_u:= \frac{\partial}{\partial u},\quad 
\partial_v:= \frac{\partial}{\partial v}.
$$
Then,
$ds^2$ is written as 
\begin{equation}\label{eq:1st-FF}
  ds^2 = E\,du^2 +2F\,du\,dv + G\,dv^2,
\end{equation}
where 
$E:=\inner{f_u}{f_u}$,
$F:=\inner{f_u}{f_v}$ and
$G:=\inner{f_v}{f_v}$.
Setting the function $\lambda$ as
$\lambda := EG-F^2$,
a point $q\in U$ is a lightlike (resp.\ spacelike, timelike) point
if and only if 
$\lambda(q)=0$ (resp.\ $\lambda(q)>0$, $\lambda(q)<0$)
holds.
We call $\lambda$ the {\it discriminant function}.

A lightlike point $p\in LD$ is called {\it non-degenerate}
if $d\lambda(p) \neq0$.
By the implicit function theorem,
there exists a regular curve
$c(t)$ $(|t|<\delta)$
in $\Sigma$
such that $p= c(0)$
and ${\rm Image}( c)=LD$
holds on a neighborhood of $p$.
We call $c(t)$ 
a {\it characteristic curve}.
Define a regular curve $\hat{c}(t)$ in $\L^3$ by 
$$
  \hat{c}(t):=f\circ c(t).
$$
Since $c(t)$ is a lightlike point for each $t$,
the image
$V_{c(t)}=df_{c(t)}(T_{c(t)}\Sigma)$
is a degenerate $2$-subspace of $\L^3$.
In particular, 
each tangent vector $\hat{c}'(t)$
cannot be timelike,
where the prime $'$ means the derivative $d/dt$.

\begin{definition}\label{def:first-kind}
Let $p\in LD$ be a non-degenerate lightlike point,
and let $c(t)$ $(|t|<\delta)$ be a characteristic curve
passing through $p=c(0)$, where $\delta>0$.
If $ \hat{c}'(0)$ is spacelike
(resp.\ lightlike),
then $p$ is said to be a \emph{lightlike point of the first kind}
(resp.\ a \emph{lightlike point of the second kind\/}).
\end{definition}

A vector field $\eta(t)$ along $c(t)$ 
is called a {\it null vector field along $c(t)$}
if 
\begin{equation}\label{eq:vf-L}
  L(t):=df(\eta(t))
\end{equation}
gives a lightlike vector field of $\L^3$ along $\hat{c}(t)$.
Then, a non-degenerate lightlike point $p\in LD$ 
is of the first kind
(resp.\ the second kind)
if and only if 
$ c'(0)$ and $\eta(0)$ are linearly independent
(resp.\ linearly dependent).

A vector field $\eta$ defined on a neighborhood of 
a non-degenerate lightlike point $p\in LD$
is called a {\it null vector field}
if the restriction 
$
  \eta(t):=\eta_{c(t)}\in T_{c(t)}\Sigma
$
gives a null vector field along $c(t)$.
Then, it was proved that
$p$ is a lightlike point of the first kind
if and only if $\eta_p \inner{df(\eta)}{df(\eta)}\neq0$
in \cite[Proposition 2.6]{HST}.

\subsection{Invariants of lightlike points}
We review here the fundamental properties of 
the invariants of the lightlike points of the first kind
introduced in \cite{HST}.

\subsubsection{Invariants of lightlike points of the first kind}
Let $f:\Sigma\to \L^3$ be a mixed type surface.
We also let $p\in \Sigma$ be a non-degenerate lightlike point,
and let $c(t)$ $(|t|<\delta)$ be a characteristic curve
passing through $p= c(0)$.
If $p$ is of the first kind,
taking sufficiently small $\delta>0$ if necessary, 
$c(t)$ $(|t|<\delta)$ consists of lightlike points of the first kind.
Then, we have that
$\hat{c}(t)=f\circ c(t)$ $(|t|<\delta)$ 
is a spacelike regular curve in $\L^3$.
The unit tangent vector field is given by
$\vect{e}(t):= \hat{c}'(t) / |\hat{c}'(t)|$.
We set the lightlike vector field $L(t)$ of $\L^3$
along $\hat{c}(t)$
as in \eqref{eq:vf-L}.
Then, there exists a
(uniquely determined) lightlike vector field $N(t)$ of $\L^3$
along $\hat{c}(t)$ such that
$$
  \inner{N(t)}{N(t)}=
  \inner{N(t)}{\vect{e}(t)}=0,\qquad
  \inner{N(t)}{L(t)}=1.
$$

\begin{definition}\label{def:invariants}
We set $\kappa_L(p)$, $\kappa_N(p)$, $\kappa_G(p)$ as
(cf.\ \cite[Definition 3.2, Lemma 3.5]{HST})
\begin{align}
\label{eq:L-singular-curvature}
  &\kappa_L(p):=
       \frac{ \inner{\hat{c}''(0)}{L(0)} }
              { \beta(0)^{1/3} |\hat{c}'(0)|^2 },\\
\label{eq:L-normal-curvature}
  &\kappa_N(p)
    := \frac{ \beta(0)^{1/3} \inner{\hat{c}''(0)}{N(0)} }
             { |\hat{c}'(0)|^2 },\\
\label{eq:L-geodesic-torsion}
  &\kappa_G(p)
    := \frac{1}{ |\hat{c}'(0)| }
       \left( \inner{L(0)}{N'(0)} 
       + \frac{\beta'(0)}{3\beta(0)} \right),
\end{align}
where $\beta(t):=\eta(t)\inner{ df(\eta) }{ df(\eta) }$.
We call 
$\kappa_L(p)$ the {\it lightlike singular curvature},
$\kappa_N(p)$ the {\it lightlike normal curvature}, and
$\kappa_G(p)$ the {\it lightlike geodesic torsion}.
\end{definition}

These invariants are introduced in \cite{HST}
to investigate the behavior of the Gaussian curvature $K$
at non-degenerate lightlike points.
In \cite{HST}, it is shown that 
these definitions are 
independent of the choice of 
parametrization $c(t)$ and 
null vector field $\eta$.
For a characteristic curve $c(t)$
which consists of lightlike points of the first kind,
we also denote the lightlike singular curvature 
$\kappa_L(c(t))$
(resp.\ the lightlike normal curvature $\kappa_N(c(t))$, 
the lightlike geodesic torsion $\kappa_G(c(t))$)
along $c(t)$,
as $\kappa_L(t)$
(resp.\ $\kappa_N(t)$, $\kappa_G(t)$),
unless otherwise noted.

\subsubsection{Lightlike singular curvature is intrinsic}

A local coordinate system
$(U;u,v)$ is called {\it adjusted} at $p$
if $\partial_v$ is a null vector at $p=(0,0)$.
Such a coordinate system is obtained 
by rotating the $uv$-plane.
If we represent $ds^2$ as \eqref{eq:1st-FF},
then $(u,v)$ is called adjusted at $p$
if and only if 
\begin{equation}\label{eq:condition-EFG}
  E(0,0)>0,\quad 
  F(0,0)=G(0,0)=0
\end{equation}
holds.
By a direct calculation,
we have the following:

\begin{lemma}\label{lem:kappa_L}
Let $f : \Sigma \to \L^3$ be a mixed type surface,
$p\in \Sigma$ be a lightlike point of the first kind,
and let $(U;u,v)$ be a coordinate system adjusted at $p=(0,0)$.
Then, we have
\begin{equation}\label{eq:kappa_L}
  \kappa_L(p) 
  = \left. \frac1{E \sqrt[3]{G_v}} \left(
   F_u - \frac1{2}E_v - 
   \frac{G_u^2}{2G_v} 
   \right)\right|_{u=v=0}.
\end{equation}
\end{lemma}

Since the adjustedness depends 
only on the first fundamental form $ds^2$,
Lemma \ref{lem:kappa_L} implies that
the lightlike singular curvature $\kappa_L$
is an intrinsic invariant
(cf.\ \cite[Lemma 3.9]{HST}).

\section{Mixed type metrics}
\label{sec:metric}

In this section, 
we shall discuss a class of type-changing metrics
called {\it admissible mixed type metrics}
(Definition \ref{def:admissible-metric}),
which is modeled on the first fundamental form
of mixed type surfaces.
We introduce 
type I and type II semidefinite points
(Definition \ref{def:typeI-II}),
and an invariant called 
the {\it intrinsic lightlike singular curvature}
$\tilde{\kappa}_L$ (Definition \ref{def:kappa_int}).
Then, we show the existence of 
a special orthogonal coordinate system
called an {\it L-coordinate system} 
(Proposition \ref{prop:L-orthogonal}).
Using such an L-coordinate system,
we investigate the asymptotic behavior 
of geodesic curvature function
at type II semidefinite points 
(Corollary \ref{cor:limiting-geod-2nd}).

\subsection{Semidefinite points of mixed type metrics}

Let $\Sigma$ be a smooth $2$-manifold.
A {\it metric} $g$ on $\Sigma$ 
is a smooth symmetric covariant tensor of rank 2.
Namely, at each point $p\in \Sigma$,
$g_p$ is a symmetric bilinear form 
on the tangent space $T_p\Sigma$,
and $\Sigma \ni q \mapsto g_q(V_q,W_q) \in \R$
is a smooth function for smooth vector fields
$V, W$ on $\Sigma$.
To fix the notations, first we introduce some terminologies.

\begin{definition}\label{def:semidefinite}
Let $g$ be a metric on a connected smooth $2$-manifold $\Sigma$,
and take a point $p\in \Sigma$.
\begin{itemize}
\item
If $g_p$ is definite,
then $p$ is called a {\it definite point}.
Similarly, if $g_p$ is indefinite 
(i.e.\ $g_p$ is neither definite nor degenerate),
then $p$ is said to be an {\it indefinite point}.
Definite and indefinite points
are also called {\it regular points}.
\item
If $g_p$ is degenerate,
then $p$ is said to be a {\it semidefinite point}, 
or a {\it singular point}.
The set of semidefinite points,
denote by $S(g)$,
is called the {\it semidefinite set}.
\item
A semidefinite point $p\in S(g)$
is called a {\it type-changing point}
if every open neighborhood $U$ of $p$ 
contains both definite and indefinite points.
\end{itemize}
\end{definition}

On a coordinate neighborhood $(U;u,v)$ of $p\in \Sigma$,
the metric $g$ is written by
\begin{equation}\label{eq:g-coord}
  g = E\,du^2 +2F\,du\,dv + G\,dv^2.
\end{equation}
We set $\lambda := EG - F^2$
and call it the {\it discriminant function}.
Then $q\in U$ is a semidefinite point 
of $g$ if and only if $\lambda(q)=0$.
A semidefinite point $p\in S(g)$ is called {\it admissible}
if $d\lambda(p)\neq0$.

\begin{definition}[Admissible mixed type metric]
\label{def:admissible-metric}
Let $\Sigma$ be a connected smooth $2$-manifold.
A  metric $g$ on $\Sigma$ 
is called a {\it mixed type metric}
if $g$ admits all three cases of
definite points, 
indefinite points, and 
semidefinite points.
A mixed type metric $g$
is called {\it admissible}
if $g$ is not negative definite on $\Sigma$, and
all semidefinite points of $g$ are admissible.
\end{definition}

Here, we give fundamental examples of mixed type metrics.

\begin{example}\label{ex:mixed-surface-metric}
Let $f: \Sigma\rightarrow \L^3$ be a mixed type surface,
where $\Sigma$ is a connected smooth $2$-manifold.
Then, the first fundamental form $ds^2$ of $f$
is a mixed type metric.
If $p\in LD$ is non-degenerate,
then there exists an open neighborhood $U$
of $p$ such that $LD\cap U$ consists of 
non-degenerate lightlike points.
Hence, $ds^2$ is an admissible mixed type metric on $U$.
In particular, if every lightlike point of $f$ is non-degenerate,
then $ds^2$ is an admissible mixed type metric on $\Sigma$.
\end{example}

\begin{example}\label{ex:2ndFF-metric}
Let $f: \Sigma\rightarrow \R^3$ be an oriented regular surface
with a unit normal vector field $\nu : \Sigma\to S^2$.
Assume that the zero set of the Gaussian curvature $K$,
$Z_K:=\{p\in \Sigma \,;\, K(p)=0\}$,
is not empty,
and $dK(p)\neq0$ holds for each $p\in Z_K$.
Then the second fundamental form $I\!I$ of $f$
is an admissible mixed type metric,
changing $\nu$ to $-\nu$ if necessary.

A concrete example is given by tori of revolution. 
Set $T^2:=S^1\times S^1$,
where $S^1:=\R/2\pi\Z$.
For $r>1$, let $f : T^2 \rightarrow \R^3$ be
the torus of revolution
given by
$$
  f(u,v):= ( (r + \cos u)\cos v,\,  (r + \cos u)\sin v,\, \sin u)^T.
$$
The second fundamental form $I\!I$ is written as
$I\!I = du^2 + (\cos u) (r +\cos u)\,dv^2$.
Since the discriminant function $\lambda$
is given by $\lambda = (r +\cos u) \cos u$,
the semidefinite set $S(I\!I)$ of $I\!I$ is
$S(I\!I)=\{ (u,v) \in T^2 \,;\, \cos u=0 \}$,
which coincides with 
the zero set $Z_K$ of the Gaussian curvature $K$.
Since $\partial\lambda/ \partial u\ne0$ holds on $S(I\!I)$,
every semidefinite point $p\in S(I\!I)$ is admissible.
And hence, the second fundamental form $I\!I$
is an admissible mixed type metric on $T^2$.
\end{example}

Let $p\in S(g)$ be a semidefinite point of 
an admissible mixed type metric $g$.
By the implicit function theorem,
there exists a regular curve 
$ c(t)$ $(|t|<\delta)$
on $\Sigma$
such that $p= c(0)$ and 
$ c(t)$ parametrizes $S(g)$ near $p$.
We call $ c(t)$ a {\it characteristic curve}.
By a proof similar to that of \cite[Lemma 2.1]{HST},
we obtain the following:

\begin{lemma}
Any semidefinite point of an admissible mixed type metric 
is a type-changing point.
\end{lemma}

At a point $p \in \Sigma$,
the subspace
$$
  \mathcal{N}_p:= \left\{
  \vect{v} \in T_p\Sigma\,;\,
  g_p(\vect{v},\vect{x})=0 \text{ holds for any } \vect{x} \in T_p\Sigma
  \right\}
$$
of $T_p\Sigma$ is called the {\it null space} at $p$.
Then, $p$ is a semidefinite point if and only if $\mathcal{N}_p\ne\{\vect{0}\}$.
A non-zero tangent vector $\vect{v} \in \mathcal{N}_p$
is called a {\it null vector} at $p$.

\begin{lemma}\label{lem:null-space}
Let $g$ be an admissible mixed type metric on $\Sigma$.
For each semidefinite point $p \in S(g)$,
the null space $\mathcal{N}_p$ 
is a $1$-dimensional subspace of $T_p\Sigma$.
\end{lemma}

\begin{proof}
If the dimension of $\mathcal{N}_p$ is $2$,
we have $E(p)=F(p)=G(p)=0$.
Since the partial derivatives of 
$\lambda=EG-F^2$ are given by
$$
\lambda_u=E_uG+EG_u - 2FF_u,\qquad
\lambda_v=E_vG+EG_v - 2FF_v,
$$
we have $\lambda_u(p)=\lambda_v(p)=0$,
which contradicts the non-degeneracy.
\end{proof}

Let $p\in S(g)$ be a semidefinite point.
A non-zero smooth vector field $\eta(t)$ 
along the characteristic curve $ c(t)$
is called a {\it null vector field along $ c(t)$} 
if $\eta(t)$ is a null vector of $T_{ c(t)}\Sigma$
for each $t$.
A non-zero smooth vector field $\eta$
defined on a neighborhood $U$ of $p$
is said to be a {\it null vector field}
if the restriction $\eta|_{S(g)}$ gives a null vector field
along the semidefinite set $S(g)$.

\begin{definition}\label{def:typeI-II}
Let $p\in S(g)$ be a semidefinite point 
of an admissible mixed type metric $g$.
Then, $p$ is said to be {\it type I} 
(resp.\ {\it type II\/}),
if $ c'(0)$ and $\eta(0)$ are linearly independent
(resp.\ linearly dependent).
\end{definition}

If $g$ is an admissible mixed type metric 
given by the first fundamental form
$ds^2$ of a mixed type surface $f:\Sigma\to \L^3$,
a point $p\in \Sigma$ is a type I semidefinite point of $g$
if and only if $p$ is a lightlike point of the first kind.

By a proof similar to that of \cite[Proposition 2.6]{HST},
we obtain the following:

\begin{proposition}\label{prop:type1}
Let $p\in S(g)$ be a semidefinite point 
of an admissible mixed type metric $g$.
On a local coordinate neighborhood $(U;u,v)$ of $p$,
set $\lambda := EG-F^2$.
Let $\eta$ be a vector field on a neighborhood of $p\in S(g)$
such that $\eta_p$ is a null vector at $p$.
Then, 
$p$ is a type I semidefinite point
if and only if $\eta \lambda(p)\neq0$.
Moreover, such the condition is 
equivalent to $\eta_p (g(\eta,\eta))\neq0$.
\end{proposition}

\subsection{Intrinsic lightlike singular curvature}

Let $p\in S(g)$ be a semidefinite point 
of an admissible mixed type metric $g$.
A local coordinate system $(U;u,v)$ centered at $p$
is called {\it adjusted at $p$} 
if $\partial_v$ gives a null vector at $p$.
Such an adjust coordinate system can be obtained 
by rotating the $uv$-plane.
As a direct corollary of Proposition \ref{prop:type1},
we have the following:

\begin{corollary}\label{cor:type1}
Let $p\in S(g)$ be a semidefinite point 
of an admissible mixed type metric $g$.
Take a local coordinate system $(U;u,v)$ adjusted at $p=(0,0)$.
Then, $p\in S(g)$ is
a type I semidefinite point if and only if $G_v(0,0)\ne0$.
\end{corollary}

We now introduce an invariant which plays 
an important role in this paper.

\begin{definition}\label{def:kappa_int}
For a type I semidefinite point $p$,
let $(U;u,v)$ be a coordinate system 
adjusted at $p=(0,0)$.
Then,
\begin{equation}\label{eq:kappa_int}
  \tilde{\kappa}_L(p) 
  := \left. \frac1{E \sqrt[3]{G_v}} \left(
   F_u - \frac1{2}E_v - 
   \frac{G_u^2}{2G_v} 
   \right)\right|_{u=v=0}
\end{equation}
is called the {\it intrinsic lightlike singular curvature} at $p$.
\end{definition}

Let $g$ be an admissible mixed type metric 
given by the first fundamental form
$ds^2$ of a mixed type surface $f:\Sigma\to \L^3$.
Then the lightlike singular curvature $\kappa_L(p)$ 
of $f$ at a lightlike point $p\in LD$ of the first kind
coincides with 
the intrinsic lightlike singular curvature $\tilde{\kappa}_L(p)$
by Lemma \ref{lem:kappa_L}.

\begin{proposition}\label{prop:ks-welldef}
The definition of intrinsic lightlike singular curvature $\tilde{\kappa}_L$
in \eqref{eq:kappa_int}
does not depend on the choice of adjusted coordinate system. 
\end{proposition}

\begin{proof}
We remark that $E(0,0)\ne0$ and $G_v(0,0)\ne0$ 
by Lemma \ref{lem:null-space} and Corollary \ref{cor:type1}.
Take coordinate neighborhoods $(U;u,v)$, $(V;x,y)$
adjusted at $p=(0,0)$. 
Denote the coordinate transformation by $(u,v)=(u(x,y),v(x,y))$.
Then
\begin{align}
\label{eq:Exy}
  \hat{E} &= E u_x^2  +2F u_x v_x +G v_x^2,\\
\label{eq:Fxy}
  \hat{F} &= E u_x u_y  +F (u_x v_y+u_y v_x) +G v_xv_y,\\
\label{eq:Gxy}
  \hat{G} &= E u_y^2  +2F u_y v_y +G v_y^2
\end{align}
hold, where we set 
$$
  g = E du^2 + 2 F du\,dv + G dv^2
  = \hat{E} dx^2 + 2\hat{F} dx\,dy + \hat{G} dy^2.
$$
Hence, we have
$F=G=\hat{F}=\hat{G}=u_y=0$
at $p=(0,0)$.
Since $(u(x,y),v(x,y))$ is a coordinate change, 
$u_x\ne0$, $v_y\ne0$ hold.
By \eqref{eq:Exy}, \eqref{eq:Fxy}, \eqref{eq:Gxy},
\begin{gather*}
  \hat{E}=Eu_x,\quad
  \hat{E}_y=2 E u_x u_{xy}+E_v u_x^2 v_y+2 F_v u_x v_x v_y+G_v v_x^2 v_y,\\
  \hat{F}_x=E u_x u_{xy}+F_u u_x^2 v_y +F_v u_x v_x v_y +G_u u_x v_x v_y +G_v v_x^2 v_y ,\\
  \hat{G}_x=G_u u_xv_y^2 +G_v v_xv_y^2 ,\quad
  \hat{G}_y=G_v v_y^3
\end{gather*}
hold at $(0,0)$.
Then, we can verify the identity
$$
  \left. \frac1{\hat{E} \sqrt[3]{\hat{G}_y}} \left(
   \hat{F}_x - \frac1{2}\hat{E}_y - 
   \frac{\hat{G}_x^2}{2\hat{G}_y} 
   \right)\right|_{x=y=0}
   =
  \left. \frac1{E \sqrt[3]{G_v}} \left(
   F_u - \frac1{2}E_v - 
   \frac{G_u^2}{2G_v} 
   \right)\right|_{u=v=0},
$$
which implies the assertion.
\end{proof}

In this paper, 
we mainly deal with the case that 
$\tilde{\kappa}_L$ does not vanish.

\begin{definition}
\label{def:generic-typeI}
Let $g$ be an admissible mixed type metric on $\Sigma$.
A type I semidefinite point $p\in S(g)$ 
is said to be {\it generic}
if the intrinsic lightlike singular curvature
$\tilde{\kappa}_L$ does not vanish at $p$.
\end{definition}

For the definition of 
genericity of type II semidefinite points,
see Definition \ref{def:generic-typeII}.

\subsection{L-coordinate system}

Fix an admissible mixed type metric $g$ defined on $\Sigma$.
Let $p\in \Sigma$ be a semidefinite point of $g$.
A regular curve
$ c(t)$ $(|t|<\delta)$
on $\Sigma$
is called a {\it non-null curve} at $p$
if $p= c(0)$ and $ c'(0)$ is not a null vector.
Here, we prove the existence of 
a special orthogonal coordinate system 
associated with a given non-null curve
(Proposition \ref{prop:L-orthogonal}).

\begin{lemma}\label{lem:L-orthogonal}
Let $p\in \Sigma$ be a semidefinite point of $g$,
and let $ c(t)$ $(|t|<\delta)$
be a non-null curve at $p$.
Then, there exists a coordinate neighborhood $(U;u,v)$ of $p$
such that 
\begin{itemize}
\item
$(u,v)$ is an orthogonal coordinate system,
that is,
$g(\partial_u,\partial_v)=0$,
\item
the image of $ c(t)$ coincides with the $u$-axis near $p=(0,0)$,
and
\item
$\partial_v$ gives a null vector field.
\end{itemize}
\end{lemma}

\begin{proof}
Let $\eta$ be a null vector field.
Set vector fields $W_1$, $W_2$
defined on a neighborhood of $p$
so that
$$
(W_1)|_{ c(t)}= c'(t),\qquad
W_2=\eta
$$
holds.
Since $(W_1)_p$ and $(W_2)_p$ are linearly independent,
there exists a coordinate neighborhood $(V; x,y)$ of $p$
such that 
$W_1$ (resp.~$W_2$) is parallel to 
$\partial_x$ (resp.~$\partial_y$)
(cf.\ \cite[Lemma B.5.4]{UYbook}).
By this construction,
the $x$-axis parametrizes the image of $c(t)$.
On $V$,
we set vector fields $Z_1$, $Z_2$
defined on a neighborhood of $p$ as
$$
  Z_1 := \partial_x,\qquad
  Z_2 := -\tilde{F}\partial_x+\tilde{E}\partial_y,
$$
where $\tilde{E}$, $\tilde{F}$ are defined by
$
  g=\tilde{E}\,dx^2 + 2\tilde{F}\,dx\,dy+\tilde{G}\,dy^2.
$
Then, $g(Z_1,Z_2)=0$ holds.
As $\partial_y$ is a null vector at $p$,
$\tilde{F}(p)=\tilde{G}(p)=0$ holds.
By Lemma \ref{lem:null-space},
we have $\tilde{E}(p)\ne0$.
And hence,
$Z_1$, $Z_2$ are linearly independent 
on a neighborhood of $p$.
So, there exists a coordinate neighborhood $(U;u,v)$ of $p$
such that 
$Z_1$ (resp.~$Z_2$) is parallel to 
$\partial_u$ (resp.~$\partial_v$).
By $g(Z_1,Z_2)=0$,
we have $g(\partial_u,\partial_v)=0$.
Since $Z_1$ is parallel to $W_1$,
the $u$-axis also parametrizes the image of $ c(t)$.
Moreover,
$\partial_v$ is parallel to $\eta$
on the semidefinite set $S(g)$.
Hence, we obtain the desired coordinate system.
\end{proof}

Let $p\in \Sigma$ be a semidefinite point of 
an admissible mixed type metric $g$.
We remark that 
we can always take a non-null curve passing through $p$.
In fact, if $(U;u,v)$
is a coordinate system centered at $p$
such that the $u$-axis coincides with 
the semidefinite set near $p=(0,0)$,
then the $u$-axis (resp.\ the $v$-axis) 
gives a non-null curve
in the case that $p$ is type I (resp.\ type II).

\begin{lemma}\label{lem:PSD}
Let $g$ be an admissible mixed type metric on $\Sigma$,
and let $p\in \Sigma$ be a semidefinite point of $g$.
Then, $g$ is positive semi-definite at $p$.
\end{lemma}

\begin{proof}
By Lemma \ref{lem:L-orthogonal},
there exists a coordinate system
$(U;u,v)$ centered at $p$ such that
$g$ is expressed as
$$
  g=E\,du^2 + G\,dv^2,\qquad
  G|_{S(g)}=0.
$$
In particular, $G(p)=0$ holds.
The discriminant function is given by
$\lambda=EG$.
Set $U_+:=\{p\in U\,;\, \lambda(p)>0 \}$.
Assume that $E(p)<0$ holds. 
Then $E<0$ on a neighborhood of $p$.
Hence, 
for $q\in U_+$ sufficiently close to $p$,
we have $E(q)<0$, $G(q)<0$,
and hence $g_q$ is negative definite,
which contradicts the admissibility of $g$.
Thus, $E(p)>0$ holds, which implies
that $g_p$ is positive semi-definite.
\end{proof}

\begin{proposition}\label{prop:L-orthogonal}
Let $p\in \Sigma$ be a semidefinite point of $g$,
and let $ c(t)$ $(|t|<\delta)$
be a non-null curve at $p$.
Then, there exists 
a coordinate neighborhood $(U;u,v)$ of $p$
such that
\begin{itemize}
\item
$(u,v)$ is an orthogonal coordinate system,
that is,
$g(\partial_u,\partial_v)=0$,
\item
the image of $ c(t)$ coincides with the $u$-axis near $p=(0,0)$,
\item
$\partial_v$ gives a null vector field, and
\item
$u\mapsto (u,0)$
gives a parametrization $ c(t)$ by arclength
$($i.e., $E(u,0)=1)$.
\end{itemize}
$($Such a coordinate system
is called an \emph{L-coordinate system associated with $ c(t)$}.$)$
\end{proposition}

\begin{proof}
By Lemmas \ref{lem:L-orthogonal} and \ref{lem:PSD},
there exists a coordinate neighborhood $(U;u,v)$
centered at a semidefinite point $p\in S(g)$
such that
$$
  g = E\,du^2+G\,dv^2,\qquad
  E>0,\qquad
  G|_{S(g)}=0,
$$
and the image of $ c(t)$ coincides with the $u$-axis near $p=(0,0)$.
We set $\omega(u,v):=\log E(u,v)$.
Then, there exists a smooth function $\omega_1(u,v)$
such that
$\omega(u,v) - \omega(u,0) = v \, \omega_1(u,v)$.
We have
$E(u,v) 
= E(u,0)\,e^{v \, \omega_1(u,v)}$.
Setting a coordinate system $(\tilde{u},\tilde{v})$ as
$$
  \tilde{u}:=\int_0^u \sqrt{E(t,0)}dt,\qquad
  \tilde{v}:=v,
$$
we have
$
  g 
  = e^{v \, \omega_1(u,v)}\,d\tilde{u}^2+G\,d\tilde{v}^2.
$
Hence,  
\begin{equation}\label{eq:L-arclength}
  g = \tilde{E}\,d\tilde{u}^2+\tilde{G}\,d\tilde{v}^2,
  \qquad
  \tilde{E}(\tilde{u},0)=1,
  \qquad
  \tilde{G}|_{S(g)}=0
\end{equation}
hold, where $\tilde{E}:=e^{v \, \omega_1(u,v)}$,
$\tilde{G}:=G$.
Thus, we obtain the desired coordinate system.
\end{proof}

Let $p\in \Sigma$ be a semidefinite point of $g$.
By Proposition \ref{prop:L-orthogonal},
there exists a coordinate system 
$(U;u,v)$ centered at $p$
such that 
\begin{itemize}
\item
$g(\partial_u,\partial_v)=0$ on $U$, and 
\item
$\partial_v$ gives a null vector field.
\end{itemize}
On $(U;u,v)$, the metric $g$ is expressed as
\begin{equation}\label{eq:L-orthogonal}
  g = E\,du^2+G\,dv^2,\qquad
  E>0,\qquad
  G|_{S(g)}=0.
\end{equation}
Such a coordinate system
is called an \emph{L-coordinate system}.

\subsection{Behavior of the geodesic curvature function}

Let $g$ be an admissible mixed type metric defined 
on a $2$-manifold $\Sigma$,
and let $p\in \Sigma$ be a type II semidefinite point of $g$.
Let $ c(t)$ $(|t|<\delta)$ be a non-null curve at $p= c(0)$.
Here, we consider the asymptotic behavior of 
geodesic curvature function $\kappa_g(t)$
along $c(t)$ which is related to the definition
of the genericity of type II semidefinite points
(cf.\ Definition \ref{def:generic-typeII}).

Proposition \ref{prop:L-orthogonal} yields that
there exists an L-coordinate system
$(U;u,v)$ associated with $ c(t)$.
That is, $ c(u)=(u,0)$, and 
\begin{equation}\label{eq:coordinate}
  g = E\,du^2+G\,dv^2,\qquad
  E(u,0)=1
\end{equation}
holds.
Since $\inner{ c'(u)}{ c'(u)}=E(u,0)=1$,
we have that $u$ is an arclength parameter.

\begin{lemma}\label{lem:geodesic-curvature}
The geodesic curvature function $\kappa_g(u)$
along $ c(u)$ $(u\neq0)$ is given by 
\begin{equation}\label{eq:geod-curv}
  \kappa_g(u) = - \frac{E_v(u,0)}{2\sqrt{|\lambda(u,0)|}}.
\end{equation}
\end{lemma}

\begin{proof}
Since $\vect{n}_g=\frac1{\sqrt{|G|}} \partial_v=\frac1{\sqrt{|\lambda|}} \partial_v$
gives the unit normal vector field along 
$ c(u)=(u,0)$ with respect to $g$,
we have
$$
\kappa_g(u)
=\inner{\nabla_{u} c'(u)}{\vect{n}_g(u)}
=\frac1{\sqrt{|\lambda|}}\Gamma_{11}^2 G
=\frac1{\sqrt{|\lambda|}}\frac{-EE_v}{2\lambda} \frac{\lambda}{E}
=\frac{-E_v}{2\sqrt{|\lambda|}},
$$
where $\Gamma_{11}^2$ is a Christoffel's symbol.
Hence, we obtain \eqref{eq:geod-curv}.
\end{proof}

On $(U;u,v)$ such that 
the metric $g$ is expressed as \eqref{eq:coordinate},
the discriminant function is given by $\lambda=EG$.
Since $p$ is type II, 
the characteristic curve is given by the graph
$u=\varphi(v)$.
As the graph $u=\varphi(v)$ passes through 
$p=(0,0)$, we have $\varphi(0)=0$.
Since the discriminant function $\lambda$ 
vanishes along the graph of $u=\varphi(v)$,
there exists a smooth function $\tilde{\lambda}(u,v)$
such that
$\lambda=(u-\varphi(v))\tilde{\lambda}(u,v)$
holds.
The non-degeneracy of $p=(0,0)$ yields 
$\tilde{\lambda}(0,0)\neq0$.
Hence, it holds that
\begin{equation}\label{eq:Lambda-2nd}
\lambda(u,v) = (u-\varphi(v))\tilde{\lambda}(u,v)
\qquad
(\tilde{\lambda}(0,0)\neq0).
\end{equation}
Substituting $v=0$ into \eqref{eq:Lambda-2nd},
and setting $\hat{\lambda}(u):=\tilde{\lambda}(u,0)$,
we have
\begin{equation}\label{eq:Lambda-2nd-2}
  \lambda(u,0) = u\,\hat{\lambda}(u)
  \qquad
  (\hat{\lambda}\neq0).
\end{equation}

\begin{proposition}\label{prop:limiting-geod-2nd}
Let $g$ be an admissible metric on $\Sigma$,
and let $p\in \Sigma$ be a type II semidefinite point.
Take a non-null curve $ c(t)$ $(|t|<\delta)$ 
passing through $p \,(= c(0))$.
Let $\kappa_g(t)$ $(t\neq0)$ be the geodesic curvature function
of $ c(t)$ $(t\neq0)$,
and let $s_g(t)$ be the arclength 
$s_g(t):=\int_0^t \sqrt{\inner{ c'(\tau)}{ c'(\tau)}} \,d\tau$
of $ c(t)$ from $p$.
Then 
$$
  \sqrt{|s_g(t)| } \kappa_g(t)
$$
can be smoothly extended across $t=0$.
\end{proposition}

\begin{proof}
Let $(U;u,v)$ be an L-coordinate system
associated with $ c(t)$.
By Lemma \ref{lem:geodesic-curvature} and 
\eqref{eq:Lambda-2nd-2}, 
there exists a smooth function $\hat{\lambda}(u)$ $(\hat{\lambda}(0)\ne0)$
such that
\begin{equation}\label{eq:kappa-g-2nd}
\sqrt{|s_g(t)| } \kappa_g(t)
=\sqrt{|u|}\frac{- E_v(u,0)}{2\sqrt{|u\,\hat{\lambda}(u)|}}
=- \frac{E_v(u,0)}{2\sqrt{|\hat{\lambda}(u)|}},
\end{equation}
which proves the assertion.
\end{proof}

We call the limit 
\begin{equation}\label{eq:limiting-g-curvature}
\mu_{ c}:=\lim_{t\rightarrow 0}\sqrt{|s_g(t)| } \kappa_g(t)
\end{equation}
the {\it limiting geodesic curvature}.

As in the case of type I semidefinite points,
we define the genericity of type II semidefinite points
as follows:

\begin{definition}
\label{def:generic-typeII}
Let $g$ be an admissible mixed type metric on $\Sigma$.
A type II semidefinite point $p\in S(g)$ 
is said to be {\it generic}
there exists a non-null curve
$c(t)$ $(|t|< \delta)$ at $p=c(0)$
such that the geodesic curvature function 
$\kappa_g(t)$ along $c(t)$ defined for $t\ne0$
is unbounded at $t=0$.
Moreover, if $g$ admits only 
generic type I semidefinite points or 
generic type II semidefinite points,
then $g$ is said to be a {\it generic mixed type metric}.
\end{definition}

By \eqref{eq:kappa-g-2nd}, 
we have
$\mu_{ c}=- E_v(0,0)/(2\sqrt{|\hat{\lambda}(0)|})$,
which implies the following:

\begin{corollary}\label{cor:limiting-geod-2nd}
Let $g$ be an admissible metric on $\Sigma$,
and let $p\in \Sigma$ be a type II semidefinite point.
Take a non-null curve $ c(t)$ $(|t|<\delta)$ 
passing through $p = c(0)$, where $\delta>0$.
Let $(U;u,v)$ be an L-coordinate system
associated with $ c(t)$ centered at $p=(0,0)$. 
Then the following conditions are mutually equivalent:
\begin{itemize}
\item[$(1)$] the geodesic curvature function $\kappa_g(t)$
along $ c(t)$ is unbounded at $p$,
\item[$(2)$] the limiting geodesic curvature $\mu_{ c}$ 
is non-zero,
\item[$(3)$] $E_v(0,0)\neq0$.
\end{itemize}
\end{corollary}

\section{Fundamental theorem of mixed type surfaces}
\label{sec:L-Gauss}

In this section,
we prove the existence of a transversal vector field
$\psi$ along a mixed type surface in $\L^3$,
which we call the {\it L-Gauss map} 
(Lemma \ref{lem:L-Gauss}).
Then,
we consider a frame associated with 
the L-Gauss map.
The compatibility condition of the frame
(Lemma \ref{lem:integrability})
yields the fundamental theorem of mixed type surfaces
(Theorem \ref{thm:fundamental}).

\subsection{L-Gauss map}
Let $f:\Sigma \to \L^3$ be a mixed type surface,
and take a non-degenerate lightlike point $p\in LD$.
As seen in Example \ref{ex:mixed-surface-metric}, 
there exists an open neighborhood $V$ of $p$
such that the first fundamental form $ds^2$
of $f$ is an admissible mixed type metric on $V$.
As a direct corollary of Proposition \ref{prop:L-orthogonal},
we have the following.

\begin{corollary}\label{cor:L-coordinate}
Let $f:\Sigma \to \L^3$ be a mixed type surface,
and take a non-degenerate lightlike point $p\in LD$.
Let $c(t)$ $(|t|<\delta)$ be a non-null curve at $p=c(0)$,
where $\delta>0$.
Then, there exists a coordinate neighborhood 
$(U;u,v)$ of $p$
such that 
\begin{itemize}
\item
$(u,v)$ is an orthogonal coordinate system,
that is, $\inner{f_u}{f_v}=0$,
\item
the image of $c(t)$ coincides with the $u$-axis near $p=(0,0)$,
\item
$\partial_v$ gives a null vector field, and
\item
$\hat{c}(u):=f(u,0)$
is a unit-speed spacelike curve in $\L^3$.
\end{itemize}
\end{corollary}

As in the case of admissible mixed type metrics,
we call such a coordinate system 
an \emph{L-coordinate system associated with $c(t)$}.
In the case that $p\in LD$ is of the first kind,
the characteristic curve $c(t)$ $(|t|<\delta)$
is a non-null curve at $p=c(0)$.
In such a case,
an L-coordinate system associated with $c(t)$
is said to be 
an {\it L-coordinate system associated 
with the lightlike set $LD$.}

Even if we do not specify the non-null curve $c(t)$,
we have the following orthogonal coordinate system:
Let $p\in \Sigma$ be a non-degenerate lightlike point of 
a mixed type surface $f:\Sigma\to \L^3$.
By Corollary \ref{cor:L-coordinate},
there exists a coordinate system 
$(U;u,v)$ centered at $p$
such that 
\begin{itemize}
\item
$\inner{\partial_u}{\partial_v}=0$ on $U$, and 
\item
$\partial_v$ gives a null vector field.
\end{itemize}
On $(U;u,v)$, the metric $ds^2$ is expressed as
\begin{equation}\label{eq:L-orthogonal-pn}
  ds^2 = E\,du^2+G\,dv^2,\qquad
  E>0,\qquad
  G|_{LD}=0.
\end{equation}
As in the case of admissible mixed type metrics,
such a coordinate system
is called an \emph{L-coordinate system}.

At each point on $LD$ near $p$,
it holds that $f_u$ is spacelike,
$f_v$ is lightlike
and $\inner{f_u}{f_v}=0$,
which satisfies the condition of 
Definition \ref{def:pn-orient}.
Hence, the L-coordinate system $(U;u,v)$
is said to be {\it p-oriented} (resp.\ {\it n-oriented\/})
if $\{f_u,f_v\}$ is p-oriented (resp.\ n-oriented)
along $LD$ near $p$.

\begin{lemma}\label{lem:L-Gauss}
Let $f: \Sigma\rightarrow \L^3$ be a mixed type surface,
$p\in \Sigma$ a non-degenerate lightlike point,
and $(\bar{U};u,v)$ an L-coordinate system
centered at $p$.
Set the sign $\ep \in \{1,-1\}$
so that $\ep=1$ $($resp.\ $\ep=-1)$
if $(u,v)$ is p-oriented $($resp.\ n-oriented$)$.
Then, there exist an open neighborhood $U\subset \bar{U}$
of $p$ and a smooth map
$\psi : U \rightarrow \R^3_1$
such that
\begin{equation}\label{eq:gaiseki-1}
  f_u\times f_v = \ep \sqrt{E}\left(f_v - G\psi\right).
\end{equation}
{\rm (}We call such a map $\psi$ an \emph{L-Gauss map}.{\rm )}
Moreover, $\psi$ satisfies  
\begin{equation}\label{eq:psi-frame}
  \inner{\psi}{\psi}=0,\quad
  \inner{\psi}{f_u}=0,\quad
  \inner{\psi}{f_v}=1.
\end{equation}
\end{lemma}

\begin{proof}
Let $q\in LD$ be a non-degenerate lightlike point.
By Fact \ref{fact:gaiseki},
$f_u(q)\times f_v(q) = \ep \sqrt{E(q)} f_v(q)$ holds,
where $\ep\in \{1,-1\}$.
By the division lemma\footnote{For example, see
\cite[Appendix A]{UY_geloma}},
there exists a smooth map $\hat{\psi}_{\ep} : U \rightarrow \L^3$
such that
$
  f_u\times f_v= \ep \sqrt{E} f_v + \lambda\,\hat{\psi}_{\ep},
$
where $U\subset \bar{U}$ is an open neighborhood of $p$.
Setting $\psi := -\ep \sqrt{E} \hat{\psi}_{\ep}$,
we have \eqref{eq:gaiseki-1}.
Since
$$
  \psi = -\frac{\sqrt{E}}{\lambda} \left( \ep f_u\times f_v - \sqrt{E} f_v \right)
$$
holds on $U\setminus LD$,
we can check \eqref{eq:psi-frame}.
By continuity, we have \eqref{eq:psi-frame} on $U$.
\end{proof}

Denote by $\Lambda^2$ the $2$-dimensional lightcone
$$
  \Lambda^2
  :=\{\vect{x}\in \L^3\,;\, 
  \inner{\vect{x}}{\vect{x}}=0,\,\vect{x}\neq \vect{0} \}.
$$
Since $\inner{\psi}{\psi}=0$ and $\psi\ne\vect{0}$,
we have that $\psi$ is a $\Lambda^2$-valued map.

\begin{lemma}\label{lem:gaiseki-other}
Let $f: \Sigma\rightarrow \L^3$ be a mixed type surface
and $\psi : U \rightarrow \Lambda^2$ an L-Gauss map
defined on an L-coordinate neighborhood 
$(U;u,v)$ centered at a non-degenerate lightlike point $p\in \Sigma$.
Then it holds that
\begin{equation}
\label{eq:gaiseki-2}
  f_u \times \psi = -\ep \sqrt{E}\psi,\quad
  f_v \times \psi = \frac{\ep }{\sqrt{E}}f_u.
\end{equation}
\end{lemma}

\begin{proof}
By \eqref{eq:vector-triplet}, we have
\begin{align*}
  f_u\times (f_u\times f_v)
  &=\inner{f_u}{f_u}f_v - \inner{f_u}{f_v}f_u =Ef_v,\\
  f_v\times (f_u\times f_v)
  &=\inner{f_v}{f_u}f_v - \inner{f_v}{f_v}f_u =-Gf_u.
\end{align*}
Then the vector product of $f_u$ (resp.\ $f_v$) and \eqref{eq:gaiseki-1}
yields \eqref{eq:gaiseki-2}.
\end{proof}

The following lemma gives a criterion 
of the orientation of L-coordinate systems.

\begin{lemma}\label{lem:detF}
Let $f: \Sigma\rightarrow \L^3$ be a mixed type surface
and $\psi : U \rightarrow \Lambda^2$ an L-Gauss map
defined on an L-coordinate neighborhood 
$(U;u,v)$ centered at a non-degenerate lightlike point 
$p\in \Sigma$.
Then, $(u,v)$ is p-oriented $($resp.\ n-oriented$)$
if and only if 
$\det (f_u,f_v,\psi)>0$ 
$($resp.\ $\det (f_u,f_v,\psi)<0)$ 
holds on $U$.
In particular, $f_u$, $f_v$ and $\psi$ 
are linearly independent at each point of $U$.
\end{lemma}

\begin{proof}
By the scalar triple product formula \eqref{eq:scalar-triplet},
we have
\begin{equation}\label{eq:frame-det}
  \det(f_u,f_v,\psi)
  =\inner{f_u \times f_v}{\psi}
  =\inner{ \ep \sqrt{E} \left(f_v - G \psi \right)}{\psi}
  =\ep  \sqrt{E},
\end{equation}
which implies the assertion.
\end{proof}

With Lemma \ref{lem:detF},
we have the linear independence of
$f_u$, $f_v$ and $\psi$. 
We call $\calF:=(f_u,f_v,\psi)$ the {\it adapted frame} along $f$.
For an L-Gauss map $\psi$,
set
$$
  X:= \inner{f_{uu}}{\psi},\qquad
  Y:= \inner{f_{uv}}{\psi},\qquad
  Z:= \inner{f_{vv}}{\psi}.
$$
Then, we call
$$
  I\!I_\psi:=X\,du^2+2Y\,du\,dv+Z\,dv^2
$$
the {\it second fundamental form associated with $\psi$}.

\begin{lemma}\label{lem:GW-matrix}
Let $f: \Sigma\rightarrow \L^3$ be a mixed type surface
and $\psi : U \rightarrow \Lambda^2$ an L-Gauss map
defined on an L-coordinate neighborhood 
$(U;u,v)$ centered at a non-degenerate lightlike point $p\in \Sigma$.
Also, let 
$I\!I_\psi$
be the second fundamental form associated with $\psi$.
Then, the adapted frame $\calF=\left( f_u,\, f_v ,\, \psi \right)$ satisfies
\begin{equation}\label{eq:GW-matrix}
  \calF_u = \calF \calU,\qquad
  \calF_v = \calF \calV,
\end{equation}
where $I\!I_\psi=X\,du^2+2Y\,du\,dv+Z\,dv^2$, and
\begin{equation}\label{eq:GW-coeff}
  \calU:= \begin{pmatrix}
    \vspace{2mm}
    \dfrac{E_u}{2E} & \dfrac{E_v}{2E} & -\dfrac{X}{E}\\
    \vspace{2mm}
    X & Y & 0\\
    -\dfrac{E_v}{2}-XG & \dfrac{G_u}{2}-YG & -Y
  \end{pmatrix},\quad
  \calV:= \begin{pmatrix}
    \vspace{2mm}
    \dfrac{E_v}{2E} & -\dfrac{G_u}{2E} & -\dfrac{Y}{E}\\
    \vspace{2mm}
    Y & Z & 0\\
    \dfrac{G_u}{2}-YG & \dfrac{G_v}{2}-ZG & -Z
  \end{pmatrix}.
\end{equation}
\end{lemma}

\begin{proof}
Since $\det\calF\neq0$ by Lemma \ref{lem:detF},
$f_{uu}$, $f_{uv}$, $f_{vv}$, $\psi_{u}$, $\psi_{v}$
are written as linear combinations of $f_u$, $f_v$, $\psi$
at each point on $U$.
By a standard method using \eqref{eq:psi-frame},
we obtain 
{\allowdisplaybreaks
\begin{align*}
  &f_{uu} = \frac{E_u}{2E} f_u + X f_v - \left( \frac{E_v}{2}+XG \right)\psi,\\
  &f_{uv} = \frac{E_v}{2E} f_u + Y f_v + \left( \frac{G_u}{2}-YG \right)\psi,\\
  &f_{vv} = - \frac{G_u}{2E} f_u + Z f_v + \left( \frac{G_v}{2}-ZG \right)\psi,\\
  &\psi_{u} = - \frac{X}{E} f_u - Y\psi,\quad
  \psi_{v} = - \frac{Y}{E} f_u - Z\psi,
\end{align*}}%
which yields \eqref{eq:GW-matrix}.
\end{proof}

\begin{lemma}\label{lem:integrability}
The integrability condition $
  \calU_v-\calV_u = \calU \calV - \calV \calU
$
of \eqref{eq:GW-matrix} is given by 
the following system of partial differential equations:
{\allowdisplaybreaks
\begin{align}
\label{eq:Cod1}\tag{C1}
  &X_v - Y_u = \frac1{2E} \left( E_v X - E_u Y \right) + Y^2 - XZ, \\
\label{eq:Cod2}\tag{C2}
  &Y_v - Z_u = -\frac1{2E} \left( G_u X +E_v Y \right),\\
\label{eq:Gauss}\tag{G}
  &E_{vv}+G_{uu} - \frac{E_u G_u +E_v^2}{2E} 
     = 2G(XZ-Y^2) - G_v X + 2 G_u Y +E_vZ.
\end{align}}
\end{lemma}

As Lemma \ref{lem:integrability} is proved by direct calculation,
and we omit the proof.
We call \eqref{eq:Gauss} the {\it Gauss equation},
and \eqref{eq:Cod1}, \eqref{eq:Cod2} the {\it Codazzi equations}.

\subsection{Fundamental theorem}

Let $f: \Sigma\rightarrow \L^3$ be a mixed type surface,
and let $p\in \Sigma$ be a non-degenerate lightlike point.
We also let $\psi : U \rightarrow \Lambda^2$ be an L-Gauss map
defined on an L-coordinate neighborhood 
$(U;u,v)$ centered at $p=(0,0)$.
Since 
$\{ \frac{1}{\sqrt{E}} f_u,\, f_v,\, \psi \}$
is a null basis at $p=(0,0)$,
there exists $T \in \SO(2,1)$ 
such that 
$$
  T\calF(0,0)
  =
  \begin{cases}
  F_0^+ & (\text{if $(u,v)$ is p-oriented}),\\
  F_0^- & (\text{if $(u,v)$ is n-oriented}),
  \end{cases}
$$
where $\calF:=(f_u,f_v,\psi)$,
and 
\begin{equation}\label{eq:Frame-ini}
  F_0^+:=
  \begin{pmatrix}
    \sqrt{E(0,0)} & 0 & 0\\
    0 &  \frac{1}{\sqrt{2}} & \frac{1}{\sqrt{2}}\\
    0 &  -\frac{1}{\sqrt{2}} & \frac{1}{\sqrt{2}}
  \end{pmatrix},
  \quad
  F_0^-:=
  \begin{pmatrix}
    \sqrt{E(0,0)} & 0 & 0\\
    0 &  -\frac{1}{\sqrt{2}} & -\frac{1}{\sqrt{2}}\\
    0 &  -\frac{1}{\sqrt{2}} & \frac{1}{\sqrt{2}}
  \end{pmatrix}.
\end{equation}
Note that $RF_0^-=F_0^+$,
where 
\begin{equation}\label{eq:R}
 R:=  \begin{pmatrix}
    1 & 0 & 0\\
    0 & -1 & 0\\
    0 &  0 & 1
  \end{pmatrix},
\end{equation}
cf.\ \eqref{eq:null-basis}.
Hence, we may conclude that there exists $T \in \O(2,1)$
such that $T\calF(0,0) = F_0^+.$
Then, we have the following:

\begin{theorem}[Fundamental theorem of mixed type surfaces]
\label{thm:fundamental}
Let $g$ be an admissible mixed type metric on a smooth $2$-manifold $\Sigma$,
and $p\in S(g)$ a semidefinite point.
Take a simply connected L-coordinate 
neighborhood $(U;u,v)$ centered at $p$
and set
\begin{equation}\label{eq:ds2}
  g = E\,du^2+G\,dv^2.
\end{equation}
Moreover, let $h$ be a symmetric $(0,2)$-tensor
\begin{equation}\label{eq:2ff}
  h=X\,du^2+2Y\,du\,dv+Z\,dv^2
\end{equation}
defined on $U$.
If $g$ and $h$ satisfy
the Gauss and Codazzi equations,
\eqref{eq:Gauss}, \eqref{eq:Cod1} and \eqref{eq:Cod2},
then there exist a mixed type surface $f : U \rightarrow \L^3$
and an L-Gauss map $\psi : U \rightarrow \Lambda^2$ such that
\begin{itemize}
\item
the first fundamental form of $f$ coincides with $g$, and
\item
the second fundamental form $I\!I$ associated with $\psi$
coincides with $h$.
\end{itemize}
Moreover, such $f$ and $\psi$
are unique up to isometries of $\L^3$.
\end{theorem}

\begin{proof}
Let $\calU$, $\calV$ be the matrices
given by \eqref{eq:GW-coeff},
where $E$, $G$, $X$, $Y$ and $Z$
are as in \eqref{eq:ds2} and \eqref{eq:2ff}.
Then, by Lemma \ref{lem:integrability},
the integrability condition of
$\calF_u = \calF \calU$,
$\calF_v = \calF \calV$
is given by the Gauss and Codazzi equations,
\eqref{eq:Gauss}, \eqref{eq:Cod1} and \eqref{eq:Cod2}.
Hence, setting $F_0^+$ as in \eqref{eq:Frame-ini},
we have that
\begin{equation}\label{eq:GW-matrix-ini}
  \calF_u = \calF \calU,\qquad
  \calF_v = \calF \calV,\qquad
  \calF(0,0)
  =F_0^+
\end{equation}
has a unique solution $\calF$.
Then, we set
$f_1(u,v):=\calF\vect{e}_1$,
$f_2(u,v):=\calF\vect{e}_2$.
By the definition \eqref{eq:GW-coeff} of 
$\calU$ and $\calV$,
we can check that
$\calV\vect{e}_1 = \calU\vect{e}_2$.
Then,
\begin{align*}
  (f_1)_v
  = \calF_v\vect{e}_1
  = \calF\calV\vect{e}_1,\quad
  (f_2)_u
  = \calF_u\vect{e}_2
  = \calF\calU\vect{e}_2
\end{align*}
implies $(f_1)_v=(f_2)_u$.
Hence, 
there exists a map $f:U\rightarrow \L^3$
such that $f_u=f_1$, $f_v=f_2$.

Set $\psi(u,v):=\calF\vect{e}_3$.
We shall prove that the first fundamental form $ds^2$ of $f$
coincides with $g$,
and $\psi$ is an L-Gauss map of $f$.
Set
$$
P:=\begin{pmatrix}
    E & 0 & 0\\
    0 & G & 1\\
    0 & 1 & 0
  \end{pmatrix},\qquad
Q:= \calF^T S \calF,
$$
where $S$ is the matrix given in \eqref{eq:S}.
By a direct calculation, we have that
$$
  Q_u
  =\calF^T_u S \calF + \calF^T S \calF_u
  =(\calF\calU)^T S \calF
    + \calF^T S \calF\calU
  =\calU^T Q + Q\calU,
$$
and similarly 
$Q_v=\calV^T Q + Q\calV$
holds.
On the other hand,
since
$$
  P\calU
  = \begin{pmatrix}
    E & 0 & 0\\
    0 & G & 1\\
    0 & 1 & 0
  \end{pmatrix}
  \begin{pmatrix}
    \frac{E_u}{2E} & \frac{E_v}{2E} & -\frac{A}{E}\\
    X & Y & Z\\
    -\frac{E_v}{2}-XG & \frac{G_u}{2}-YG & -Y
  \end{pmatrix}
  =
  \begin{pmatrix}
    \frac{E_u}{2} & \frac{E_v}{2} & -X\\
    -\frac{E_v}{2} & \frac{G_u}{2} & -Y\\
    X & Y & 0
  \end{pmatrix},
$$
we have
$
  \calU^T P + P\calU
  = (P\calU)^T + P\calU
  = P_u.
$
Similarly, we can check that
$P_v=\calV^T P + P\calV$ holds.
Moreover, by
$$
  Q(u_0,v_0)
  = \calF(u_0,v_0)^T S \calF(u_0,v_0)
  = \begin{pmatrix}
    E(u_0,v_0) & 0 & 0\\
    0 & 0 & 1\\
    0 & 1 & 0
  \end{pmatrix}
  = P(u_0,v_0),
$$
$P, Q$ are solutions to the following initial value problem
$$
  \calG _u=\calU^T\calG  + \calG \calU,\quad
  \calG _v=\calV^T\calG  + \calG \calV,\quad
  \calG (u_0,v_0)
  = \begin{pmatrix}
    E(u_0,v_0) & 0 & 0\\
    0 & 0 & 1\\
    0 & 1 & 0
  \end{pmatrix}.
$$
Hence, we have $P=Q$ by the uniqueness.
Since $\inner{\vect{x}}{\vect{y}}=\vect{x}^T S \vect{y}$
for any $\vect{x}, \vect{y}\in \L^3$,
we have
$$
  Q
  = \calF^T S \calF
  =
  \begin{pmatrix}
  f_u^T \\
  f_v^T \\
  \psi^T 
  \end{pmatrix}
   S \left( f_u,\, f_v,\, \psi \right)
  =
  \begin{pmatrix}
  \inner{f_u}{f_u} &  \inner{f_u}{f_v} & \inner{f_u}{\psi} \\
  \inner{f_v}{f_u} &  \inner{f_v}{f_v} & \inner{f_v}{\psi}  \\
  \inner{\psi}{f_u} &  \inner{\psi}{f_v} & \inner{\psi}{\psi}
  \end{pmatrix}.
$$
So, by $P=Q$, we have
$$
  \inner{f_u}{f_u}=E,\quad
  \inner{f_v}{f_v}=G,\quad
  \inner{f_v}{\psi}=1,\quad
  \inner{f_u}{f_v}=\inner{f_u}{\psi}=\inner{\psi}{\psi}=0.
$$
Moreover, using $\calF^T S \calF= P$,
we have
\begin{align*}
  \inner{f_{uu}}{\psi}
  = \inner{\calF_u\vect{e}_1}{\calF \vect{e}_3}
  &= \inner{\calF\calU\vect{e}_1}{\calF \vect{e}_3}
  = (\calF\calU\vect{e}_1)^T S\calF\vect{e}_3\\
  &=  
  \left(\dfrac{E_u}{2E} ,\, X,\,  -\dfrac{E_v}{2}-XG \right)
  \begin{pmatrix}
    E & 0 & 0\\
    0 & G & 1\\
    0 & 1 & 0
  \end{pmatrix}
  \begin{pmatrix}
    0 \\  0\\ 1
  \end{pmatrix}
  =X.
\end{align*}
Similarly, it holds that
$\inner{f_{uv}}{\psi}=Y$, 
$\inner{f_{vv}}{\psi}=Z$.
Thus, the second fundamental form 
$I\!I_{\psi}$ with respect to $\psi$
coincides with $h$.

Finally, with respect to the uniqueness,
let $f,\,\check{f} : U \rightarrow \L^3$
be mixed type surfaces
such that the first fundamental forms
of $f,\,\check{f}$ are both \eqref{eq:ds2},
and let $\psi,\,\check{\psi} : U \rightarrow \Lambda^2$
be the L-Gauss maps of $f,\,\check{f}$, 
respectively,
such that 
both $I\!I$ and $\check{I\!I}$
coincide with \eqref{eq:2ff},
where 
$I\!I$ (resp.\ $\check{I\!I}$)
is the second fundamental form of $f$ (resp.\ $\check{f}$)
with respect to $\psi$ (resp.\ $\check{\psi}$).
We set $\calF:=(f_u,f_v,\psi)$,
$\check{\calF}:=(\check f_u, \check f_v, \check\psi)$.
Applying suitable isometries of $\L^3$
to $f$, $\check f$,
we may assume that 
$\calF(u_0,v_0)=\check{\calF}(u_0,v_0)=F_0^+$,
where  $F_0^+$ is given by \eqref{eq:Frame-ini}.
Then, both $\calF$ and $\check{\calF}$
are solutions to the initial value problem
\eqref{eq:GW-matrix-ini}.
Hence, by the uniqueness, we have $\calF=\check{\calF}$,
which gives the desired result.
\end{proof}

By a proof similar to that of Theorem \ref{thm:fundamental},
we have the following.

\begin{corollary}
\label{cor:fundamental}
Let $g$ be an admissible mixed type metric 
on a smooth $2$-manifold $\Sigma$,
and $p\in S(g)$ a semidefinite point.
Moreover, let $c(t)$ $(|t|<\delta)$
be a non-null curve in $\Sigma$ passing through $p=c(0)$,
where $\delta>0$.
Take a simply connected L-coordinate 
system $(U;u,v)$ centered at $p=(0,0)$,
which is associated with $c(t)$, 
and set $g = E\,du^2+G\,dv^2$.
In addition, let $h$ be a symmetric $(0,2)$-tensor
$h=X\,du^2+2Y\,du\,dv+Z\,dv^2$
defined on $U$,
and take a null basis $\{ \vect{w}_1,\vect{w}_2,\vect{w}_3 \}$.
If $g$ and $h$ satisfy
the Gauss and Codazzi equations,
\eqref{eq:Gauss}, \eqref{eq:Cod1} and \eqref{eq:Cod2},
then there exist a unique mixed type surface 
$f : U \rightarrow \L^3$
and an L-Gauss map $\psi : U \rightarrow \Lambda^2$ such that
\begin{itemize}
\item
the first fundamental form $ds^2$ of $f$ coincides with $g$,
\item
the second fundamental form $I\!I_{\psi}$ associated with $\psi$
coincides with $h$,
\item
$f(0,0)=\vect{0}$, and
$\calF(0,0)=\left( \vect{w}_1,\vect{w}_2,\vect{w}_3 \right)$.
\end{itemize}
\end{corollary}

\section{Spacelike curves on mixed type surfaces in $\L^3$}
\label{sec:curve}

As in the introduction,
a mixed type surface is said to be {\it generic}
if its first fundamental form is 
a generic mixed type metric,
cf.\ Definitions \ref{def:generic-surface},
\ref{def:generic-typeI} and \ref{def:generic-typeII}.
In this section, 
as a preparation for the proof of Theorem \ref{thm:main},
we calculate the invariants,
such as the curvature and torsion functions,
of spacelike curves on generic mixed type surfaces.

\subsection{Spacelike curves in $\L^3$}

Here, we review the fundamental properties 
of spacelike curves in $\L^3$.
See \cite{Walrave, Lopez, Honda-Lk} for details.

Let $I$ be an open interval.
A regular curve
$\gamma : I \rightarrow \L^3$
is called {\it spacelike} 
if each tangent vector is spacelike.
By a coordinate change,
we may assume that
$\gamma$ is parametrized by arclength.
That is, 
$\vect{e}(u):=\gamma'(u)$
gives the unit spacelike tangent vector field,
where the prime means $d/du$.
We call 
$$
  \vect{\kappa}(u):=\gamma''(u)
$$
the {\it curvature vector field} along $\gamma(u)$.
If $\vect{\kappa}(u)$ is nowhere zero, 
then $\gamma(u)$ is said to have {\it non-zero curvature}.
For a general parametrization,
$\gamma(t)$ has non-zero curvature
if and only if 
$d\gamma/dt$ and $d^2\gamma/dt^2$ are linearly independent
\cite[Lemma 2.2]{Honda-Lk}.

From now on, we assume that 
every spacelike curve is parametrized by arclength.
To measure the causal character of 
the curvature vector field $\vect{\kappa}(u)$,
set
$$
\theta(u):=
\inner{\vect{\kappa}(u)}{\vect{\kappa}(u)}
~(=\inner{\gamma''(u)}{\gamma''(u)}).
$$
We call $\theta(u)$ the {\it causal curvature function} 
along $\gamma(u)$.

\begin{definition}\label{def:non-Frenet}
Let $\gamma : I \rightarrow \L^3$
be a unit-speed spacelike curve 
with non-zero curvature
$\vect{\kappa}(u) \ne \vect{0}$.
Let $\theta(u)$ be the causal curvature function,
and let $k\in \Z$ be a positive integer.
Then, $\gamma(u)$ is said to be
\begin{itemize}
\item
of {\it type $S$}
if $\vect{\kappa}(u)$ is a spacelike vector field, 
that is $\theta(u)>0$ holds on $I$,
\item
of {\it type $T$}
if $\vect{\kappa}(u)$ is a timelike vector field,
that is $\theta(u)<0$ holds on $I$,
\item 
of {\it type $L$}
if $\vect{\kappa}(u)$ is a lightlike vector field,
that is $\theta(u)=0$ holds on $I$,
\item 
of {\it type $L_k$} at $u_0\in I$
if $\theta(u)$ satisfies the following:
$$
  \theta(u_0)=\theta'(u_0)= \cdots= \theta^{(k-1)}(u_0)=0, \qquad
  \theta^{(k)}(u_0)\neq0.
$$
\end{itemize}
If the curve $\gamma(u)$ is of type $S$ or type $T$,
then $\gamma(u)$ is said to be a {\it Frenet curve} \cite{Lopez}.
In the case of type $L$ or type $L_k$,
we call $\gamma(u)$ a {\it non-Frenet curve}.
\end{definition}

We remark that, every real analytic 
spacelike curve $\gamma(u)$
with non-zero curvature 
must be either a Frenet curve
(i.e.\ type $S$ or type $T$)
or a non-Frenet curve 
(i.e.\ type $L$ or type $L_k$).

Here, we briefly review the fundamental properties of spacelike curves 
of type $S$, $T$, $L$ and $L_k$.
We fix a spacelike curve
$\gamma : I \to \L^3$
which is parametrized by arclength.
We denote the unit tangent vector field 
(resp.\ the curvature vector field) 
along $\gamma(u)$ by $\vect{e}(u)=\gamma'(u)$
$\left({\rm resp}.\ \vect{\kappa}(u)=\gamma''(u) \right)$.

\subsubsection{Spacelike Frenet curves}
\label{sec:type-ST}

In the case that $\gamma$ is a Frenet curve,
the function
$\kappa(u):=\sqrt{|\theta(u)|}\,(=|\vect{\kappa}(u)|)$
is called the curvature function \cite{Lopez}.
However, for the purpose of unified treatment,
we use $\theta(u)$ instead of $\kappa(u)$.
Then, the function
\begin{equation}\label{eq:torsion-det}
  \tau = -\frac1{\theta}\det (\gamma' ,\, \gamma'',\, \gamma''')
\end{equation}
is called the \emph{torsion function}.
Moreover, 
the principal normal vector field $\vect{n}(u)$ and 
the binormal vector field $\vect{b}(u)$
are defined as
$$
  \vect{n}(u):=\frac1{\sqrt{|\theta(u)|}}\vect{\kappa}(u),\qquad
  \vect{b}(u):= \sigma\vect{e}(u) \times \vect{n}(u),
$$
respectively,
where 
$$
  \sigma
  := \begin{cases}
    -1 & (\text{if $\gamma(u)$ is of type $S$}), \\
    1 & (\text{if $\gamma(u)$ is of type $T$}).
    \end{cases}
$$
By the fundamental theorem for spacelike Frenet curves
(cf.\ \cite[Theorems 2.6, 2.8]{Lopez}),
the following holds.

\begin{proposition}\label{prop:fund-ST}
Let $I$ be an open neighborhood of the origin $0$,
and let $\gamma(u), \bar{\gamma}(u) : I \rightarrow \L^3$
be two unit-speed spacelike Frenet curves
such that $\gamma(0)=\bar{\gamma}(0)$.
Let $\theta$, $\tau$ 
$($resp.\ $\bar{\theta}$, $\bar{\tau})$
be the causal curvature and torsion function
of $\gamma$ $($resp.\ $\bar{\gamma})$, respectively.
Moreover, we let 
$\vect{e}$, $\vect{n}$, $\vect{b}$
$($resp.\ $\bar{\vect{e}}$, $\bar{\vect{n}}$, $\bar{\vect{b}})$
be the unit tangent vector field,
the principal normal vector field, and 
the binormal vector field along $\gamma(u)$
$($resp.\ $\bar{\gamma}(u))$, respectively.
Then, $\gamma(u)=\bar\gamma(u)$ holds
if and only if 
$\theta(u)=\bar{\theta}(u)$, $\tau(u)=\bar\tau(u)$,
and 
$(\vect{e}(0), \vect{n}(0), \vect{b}(0))=
(\bar{\vect{e}}(0), \bar{\vect{n}}(0), \bar{\vect{b}}(0))$
hold.
\end{proposition}

\subsubsection{Spacelike curves of type $L$}
\label{sec:type-L}

Let $\gamma : I \to \L^3$ be a spacelike curve of type $L$.
By definition,
$\vect{\kappa}(u)$ is a lightlike vector field.
By Fact \ref{fact:gaiseki},
$\vect{e}(u)\times\vect{\kappa}(u) = \epsilon \vect{\kappa}(u)$
holds for $\epsilon\in \{1,-1\}$.
Such $\epsilon$ is called the {\it signature} of $\gamma$.
If $\vect{e}(u)\times\vect{\kappa}(u) = \vect{\kappa}(u)$
(resp.\ $\vect{e}(u)\times\vect{\kappa}(u) = -\vect{\kappa}(u)$)
holds,
then $\gamma$ is called of type $L^+$ (resp.\ type $L^-$).
Let $\vect{\beta}(u)$ be the vector field
which satisfies
$$
  \inner{\vect{\beta}(u)}{\vect{\beta}(u)}
  = \inner{\vect{e}(u)}{\vect{\beta}(u)}
  =0,\qquad
  \inner{\vect{\kappa}(u)}{\vect{\beta}(u)}=1.
$$
Such a vector field $\vect{\beta}(u)$ is uniquely determined.
We call $\vect{\beta}(u)$ the {\it pseudo-binormal vector field}.
Then, $\mu(u) := - \inner{\vect{\kappa}'(u)}{\vect{\beta}(u)}$ 
is called the {\it pseudo-torsion function}.
The Frenet-Serret type formula is given as
\begin{equation}\label{eq:Frenet-L}
  \vect{e}' = \vect{\kappa},\qquad
  \vect{\kappa}' = -\mu \vect{\kappa},\qquad
  \vect{\beta}' = - \vect{e} + \mu \vect{\beta}.
\end{equation}
Set $\calC:=(\vect{e},\vect{\kappa},\vect{\beta})$.
By \eqref{eq:scalar-triplet},
$
  \det\calC=
  \inner{\vect{e}\times\vect{\kappa}}{\vect{\beta}}
  = \epsilon
$
holds.
Hence, $\det\calC>0$
(resp.\ $\det\calC<0$)
if and only if 
$\gamma$ is of type $L^+$
(resp.\ type $L^-$).

By the fundamental theorem for spacelike curves of type $L$
(cf.\ \cite[Theorems 2.7, 2.8]{Lopez}),
the following holds.

\begin{proposition}\label{prop:fund-L}
Let $I$ be an open neighborhood of the origin $0$,
and let $\gamma(u), \bar{\gamma}(u) : I \rightarrow \L^3$
be two unit-speed spacelike curves of type $L$
such that $\gamma(0)=\bar{\gamma}(0)$.
Let $\mu$ $($resp.\ $\bar{\mu})$
be the pseudo-torsion function
of $\gamma$ $($resp.\ $\bar{\gamma})$.
Moreover, we let 
$\vect{e}$, $\vect{\kappa}$, $\vect{\beta}$
$($resp.\ $\bar{\vect{e}}$, $\bar{\vect{\kappa}}$, $\bar{\vect{\beta}})$
be the unit tangent vector field,
the curvature vector field, and 
the pseudo-binormal vector field along $\gamma(u)$
$($resp.\ $\bar{\gamma}(u))$, respectively.
Then, $\gamma(u)=\bar\gamma(u)$ holds
if and only if 
$\mu(u)=\bar\mu(u)$, and 
$(\vect{e}(0), \vect{\kappa}(0), \vect{\beta}(0))=
(\bar{\vect{e}}(0), \bar{\vect{\kappa}}(0), \bar{\vect{\beta}}(0))$
hold.
\end{proposition}

We remark that, by the condition 
$(\vect{e}(0), \vect{\kappa}(0), \vect{\beta}(0))=
(\bar{\vect{e}}(0), \bar{\vect{\kappa}}(0), \bar{\vect{\beta}}(0))$,
the signature $\epsilon$ of $\gamma(u)$ coincides with 
that $\bar{\epsilon}$ of $\bar{\gamma}(u)$.

\subsubsection{Spacelike curves of type $L_k$}
\label{sec:Lk}

Let $\gamma : I \to \L^3$ be a spacelike curve of type $L_k$
at $u_0\in I$.
By definition,
$\vect{\kappa}(u_0)$ is a lightlike vector.
By Fact \ref{fact:gaiseki},
$\vect{e}(u_0)\times\vect{\kappa}(u_0) = \epsilon \vect{\kappa}(u_0)$
holds for $\epsilon\in \{1,-1\}$.
Such $\epsilon$ is called the {\it signature} of $\gamma$.
If $\vect{e}(u_0)\times\vect{\kappa}(u_0) = \vect{\kappa}(u_0)$
(resp.\ $\vect{e}(u_0)\times\vect{\kappa}(u_0) = -\vect{\kappa}(u_0)$)
holds,
then $\gamma$ is called of type $L_k^+$ (resp.\ type $L_k^-$) 
at $u_0\in I$.
As seen in \cite[Corollary 3.4]{Honda-Lk},
$$
  \vect{\beta}(u):=\frac1{\theta}
  \left( \epsilon \vect{e}(u) \times \vect{\kappa}(u)
  -\vect{\kappa}(u) \right)
$$
can be smoothly extended across $u_0\in I$.
Such the vector field $\vect{\beta}(u)$ 
is called the {\it pseudo-binormal vector field}.
It satisfies 
(cf.\ \cite[Proposition 3.5]{Honda-Lk})
$$
  \inner{\vect{\beta}(u)}{\vect{\beta}(u)}
  = \inner{\vect{e}(u)}{\vect{\beta}(u)}
  =0,\qquad
  \inner{\vect{\kappa}(u)}{\vect{\beta}(u)}=1.
$$
Then, $\mu(u) := - \inner{\vect{\kappa}'(u)}{\vect{\beta}(u)}$ 
is called the {\it pseudo-torsion function}.
The Frenet-Serret type formula is given as
\begin{equation}\label{eq:Frenet-Lk}
  \vect{e}' = \vect{\kappa},\qquad
  \vect{\kappa}' = -\theta \vect{e} -\mu \vect{\kappa}
  +\left( \theta\mu +\frac1{2}\theta' \right)\vect{\beta},\qquad
  \vect{\beta}' = - \vect{e} + \mu \vect{\beta}.
\end{equation}
Set $\calC:=(\vect{e},\vect{\kappa},\vect{\beta})$.
By \eqref{eq:scalar-triplet},
$
  \det\calC= \epsilon
$
holds.
Hence, $\det\calC>0$
(resp.\ $\det\calC<0$)
if and only if 
$\gamma$ is of type $L_k^+$
(resp.\ type $L_k^-$).

By the fundamental theorem for spacelike curves
of type $L_k$
(cf.\ \cite[Lemma 3.7, Theorem 3.8]{Honda-Lk}),
the following holds.

\begin{proposition}\label{prop:fund-Lk}
Let $I$ be an open neighborhood of the origin $0$,
and let $\gamma(u), \bar{\gamma}(u) : I \rightarrow \L^3$
be two unit-speed spacelike curves
of type $L_k$ at $u=0$,
such that $\gamma(0)=\bar{\gamma}(0)$.
Let $\theta$, $\mu$ 
$($resp.\ $\bar{\theta}$, $\bar{\mu})$
be the causal curvature and pseudo-torsion function
of $\gamma$ $($resp.\ $\bar{\gamma})$, respectively.
Moreover, we let 
$\vect{e}$, $\vect{\kappa}$, $\vect{\beta}$
$($resp.\ $\bar{\vect{e}}$, $\bar{\vect{\kappa}}$, $\bar{\vect{\beta}})$
be the unit tangent vector field,
the curvature vector field, and 
the pseudo-binormal vector field along $\gamma(u)$
$($resp.\ $\bar{\gamma}(u))$, respectively.
Then, $\gamma(u)=\bar\gamma(u)$ holds
if and only if 
$\theta(u)=\bar{\theta}(u)$, $\mu(u)=\bar\mu(u)$,
and 
$(\vect{e}(0), \vect{\kappa}(0), \vect{\beta}(0))=
(\bar{\vect{e}}(0), \bar{\vect{\kappa}}(0), \bar{\vect{\beta}}(0))$
hold.
\end{proposition}

As in the case of spacelike curves of type $L$,
the condition 
$(\vect{e}(0), \vect{\kappa}(0), \vect{\beta}(0))=
(\bar{\vect{e}}(0), \bar{\vect{\kappa}}(0), \bar{\vect{\beta}}(0))$
implies that
the signature $\epsilon$ of $\gamma(u)$ coincides with 
that $\bar{\epsilon}$ of $\bar{\gamma}(u)$.

\subsection{Spacelike curves on generic mixed type surfaces} 
Here, we calculate the invariants
(such as the causal curvature, the torsion, and the pseudo-torsion) 
of spacelike curves on generic mixed type surfaces.

As in Definitions \ref{def:generic-surface} and \ref{def:generic-typeI},
a lightlike point $p$ of the first kind is said to be generic if 
the lightlike singular curvature $\kappa_L$
does not vanish at $p$.

\begin{proposition}\label{prop:STL}
Let $f: \Sigma \rightarrow \L^3$ be a mixed type surface,
and let $p\in \Sigma$ be a generic lightlike point of the first kind.
Let $(U;u,v)$ be an L-coordinate system
associated with the lightlike set $LD$ such that
$$
  ds^2 = E\,du^2 + G\,dv^2,\quad
  E(u,0)=1,\quad
  G(u,0)=0
$$
hold.
Let $\psi$ be an L-Gauss map 
and set 
$$
  x(u):=X(u,0),\qquad 
  y(u):=Y(u,0),
$$
where
$X:=\inner{f_{uu}}{\psi}, Y:=\inner{f_{uv}}{\psi}$.
Then $\hat{c}(u):=f(u,0)$ is a unit-speed spacelike curve
of non-zero curvature.
Moreover, $\theta(u)$ is the causal curvature function
of $\hat{c}(u)$ if and only if
\begin{equation}\label{eq:a-1st}
  x(u) = -\frac{\theta(u)}{ E_v(u,0) }
\end{equation}
holds.
Furthermore, set $\calF:=(f_u,f_v,\psi)$
and $\ep\in \{1,-1\}$ as
$\ep=1$ $($resp.\ $\ep=-1)$
if the L-coordinate system $(u,v)$ is 
p-oriented $($resp.\ n-oriented$)$.
Then we have the following:
\begin{itemize}
\item
Consider the case that $\hat{c}(u)$ is a Frenet curve.
Then, $\tau(u)$ is the torsion function of $\hat{c}(u)$
if and only if
\begin{equation}\label{eq:tau-ST-1}
  y(u)  = \ep\tau(u) 
  +\frac{E_{uv}(u,0)}{E_v(u,0)} - \frac{\theta'(u)}{2\theta(u)}.
\end{equation}
Moreover, it holds that 
\begin{equation}\label{eq:frame0-ST-1}
\calF(0,0)=
  \left( \vect{e}(0),
    \frac{\sqrt{|\theta(0)|}}{2a(0)}(\vect{n}(0)+\ep \sigma\vect{b}(0) ), 
    \frac{\sqrt{|\theta(0)|}}{E_v(0,0)}(-\vect{n}(0)+\ep \sigma\vect{b}(0) ) 
  \right),
\end{equation}
where $\vect{e}(u)$ $($resp.\ $\vect{n}(u)$, $\vect{b}(u))$
is the unit tangent vector field
$($resp.\ the principal normal vector field, the binormal vector field$)$ 
along $\hat{c}(u)$,
and $\sigma=-1$ $($resp.\ $\sigma=1)$
if $\hat{c}(u)$ is of type $S$ $($resp.\ type $T)$.
\item
Consider the case that $\hat{c}(u)$ is a non-Frenet curve.
Then, the coordinate system $(u,v)$ is 
p-oriented $($resp.\ n-oriented$)$
if and only if $\hat{c}(u)$ is 
of type $L^-$ or $L_k^-$ 
$($resp.\ type $L^+$ or $L_k^+)$.
Moreover, $\mu(u)$ is the pseudo-torsion function of $\hat{c}(u)$
if and only if
$y(u)$ is written as
\begin{equation}\label{eq:tau-L-1}
  y(u) = \mu(u)  +\frac{E_{uv}(u,0)}{E_v(u,0)}.
\end{equation}
Furthermore, it holds that 
\begin{equation}\label{eq:frame0-L-1}
\calF(0,0)=
  \left( \vect{e}(0),
    -\frac{E_v(0,0)}{2} \vect{\beta}(0),
    -\frac{2}{E_v(0,0)} \vect{\kappa}(0)
  \right),
\end{equation}
where $\vect{e}(u)$ $($resp.\ $\vect{\kappa}(u)$, $\vect{\beta}(u))$
is the unit tangent vector field
$($resp.\ the curvature vector field, the pseudo-binormal vector field$)$ 
along $\hat{c}(u)$.
\end{itemize}
\end{proposition}

To prove Proposition \ref{prop:STL},
we prepare the following lemma:

\begin{lemma}\label{lem:kappa-L-N}
Let $f: \Sigma \rightarrow \L^3$ be a mixed type surface,
and let $p\in \Sigma$ be a lightlike point of the first kind.
Let $(U;u,v)$ be an L-coordinate system
associated with the lightlike set at $p$.
Set $ds^2 = E\,du^2 + G\,dv^2$.
Then, the lightlike singular curvature $\kappa_L$ 
is written as
\begin{equation}\label{eq:kappa-ell-uaxis}
  \kappa_L(u) 
  = - \frac{ E_v(u,0)}{2\sqrt[3]{G_v(u,0)} }.
\end{equation}
In particular, $p\in \Sigma$ is 
a generic lightlike point of the first kind 
if and only if $E_v(0,0)\neq0$.
\end{lemma}

\begin{proof}
The lightlike set image is given by
$\hat{c}(u)=f(u,0)$.
Since
$\inner{\hat{c}'(u)}{\hat{c}'(u)}=E(u,0)=1$,
we have that $\hat{c}(u)$ is a spacelike curve 
parametrized by arclength.
By Lemma \ref{lem:kappa_L},
we have that \eqref{eq:kappa-ell-uaxis}.
\end{proof}

\begin{proof}[Proof of Proposition \ref{prop:STL}]
By Lemma \ref{lem:GW-matrix},
\begin{equation}\label{eq:GW-sub}
  f_{uu}(u,0) = x f_v -\frac{E_v}{2}\psi,\quad
  f_{uv}(u,0) \equiv y f_v,\quad
  \psi_{u}(u,0) \equiv -y\psi
\end{equation}
holds along the $u$-axis,
where $\vect{a} \equiv \vect{b}$
if  $\vect{a}- \vect{b}$ is parallel to $\vect{e}(u)=f_u(u,0)$.
Then,
the causal curvature function
$\theta(u)$ of $\hat{c}(u)$
is written as
$$
  \theta 
  = \inner{f_{uu}(u,0)}{f_{uu}(u,0)}
  = -E_v x.
$$
Since $E_v(0,0)\ne0$ holds by Lemma \ref{lem:kappa-L-N},
we obtain \eqref{eq:a-1st}.

\medskip
\noindent
{\it Case I {\rm :} $\hat{c}(u)$ is a Frenet curve.}
By \eqref{eq:GW-sub},
we have
$\hat{c}'(u)=f_u(u,0)$,
$\hat{c}''(u) = f_{uu}(u,0) = - x f_v +\frac{E_v}{2}\psi$,
and 
$$
  \hat{c}'''(u)
  =f_{uuu}(u,0)
  \equiv (x' + xy) f_v + \frac{1}{2}(E_{v}y - E_{uv})\psi.
$$
Then, it holds that
\begin{align*}
  \det(\hat{c}',\,\hat{c}'',\,\hat{c}''')
  &= \left( E_{v}xy - \frac1{2}E_{uv}x + \frac1{2}E_{v}x'  \right)
           \det\left( f_u,\, f_v,\, \psi \right)\\
  &= \ep\left( - \theta y - \frac1{2}E_{uv}x + \frac1{2}E_{v}x'  \right),
\end{align*}
where 
$\ep=\det\left( f_u(u,0),\, f_v(u,0),\, \psi(u,0) \right) 
\in \{1,-1\}$
(cf.\ Lemma \ref{lem:detF}).
On the other hand, \eqref{eq:torsion-det} yields that 
$ - \theta \tau 
= \det(\hat{c}',\,\hat{c}'',\,\hat{c}''')$,
which implies \eqref{eq:tau-ST-1}.
With respect to \eqref{eq:frame0-ST-1},
we have $\vect{e}(0)=f_u(0,0)$,
$$
  \vect{n}(0)
  =\left.\frac1{\sqrt{|\theta|}}f_{uu}\right|_{(u,v)=(0,0)}
  =\left.\frac1{\sqrt{|\theta|}}\left(xf_v-\frac{E_v}{2}\psi\right)\right|_{(u,v)=(0,0)}
$$
and 
$$
  \vect{b}(0)
  =\sigma \vect{e}(0) \times \vect{n}(0)
  =\left.\frac{\ep\sigma}{\sqrt{|\theta|}}
  \left(x f_v+\frac{E_v}{2}\psi\right)\right|_{(u,v)=(0,0)},
$$
where we applied Lemma \ref{lem:gaiseki-other}.
Thus, we obtain \eqref{eq:frame0-ST-1}.

\medskip
\noindent
{\it Case II {\rm :} $\hat{c}(u)$ is of type $L$.}
Since $\theta(u)=0$, we have $x(u)=0$.
By \eqref{eq:GW-sub},
the curvature vector field
$\vect{\kappa}(u)=\hat{c}''(u)$
is written as
$\vect{\kappa}(u) = -\frac{E_v}{2}\psi(u,0)$.
Hence, 
the pseudo-binormal vector field
$\vect{\beta}(u)$ is parallel to $f_v(u,0)$.
Since $\inner{\vect{\beta}(u)}{\vect{\kappa}(u)}=1$,
we have $\vect{\beta}(u) = -\frac{2}{E_v}f_v$.
The derivative $\vect{\beta}'(u)$
can be calculated as
$$
  \vect{b}'(u) \equiv 
  \left( y-\frac{E_{uv}}{E_v} \right)\vect{b}(u).
$$
Since $\vect{b}'(u)\equiv \mu(u)\vect{b}(u)$ 
by \eqref{eq:Frenet-L},
we have \eqref{eq:tau-L-1}.
Moreover, since
$$
   \det(\vect{e},\,\vect{\kappa},\,\vect{\beta})
   = \det\left( f_u\,-\frac{E_v}{2}\psi,\,-\frac{2}{E_v}f_v \right)
   = -\det\left( f_u\,f_v,\,\psi \right)
   =-\ep,
$$
we have the desired result.
Since $\vect{e}(0)=f_u(0,0)$, and 
\begin{equation}\label{eq:frame0-L}
  \vect{\kappa}(0) = -\frac{E_v(0,0)}{2}\psi(0,0),\qquad
  \vect{\beta}(0) =-\frac{2}{E_v(0,0)}f_v(0,0),
\end{equation}
we obtain \eqref{eq:frame0-L-1}.

\medskip
\noindent
{\it Case III {\rm :} $\hat{c}(u)$ is of type $L_k$.}
By \eqref{eq:GW-sub},
the curvature vector field
$\vect{\kappa}(u)=\hat{c}''(u)$
is written as
$\vect{\kappa}(u)
= x f_v - \frac{E_v}{2}\psi(u,0).$
Since the pseudo-binormal vector field
$\vect{\beta}(u)$
is orthogonal to $\hat{c}'(u)$,
we may set
$
  \vect{\beta}(u) = P(u)\,f_v + Q(u)\,\psi,
$
where $P(u)$, $Q(u)$ are functions.
Then, the conditions 
$\inner{\vect{\beta}(u)}{\vect{\kappa}(u)}=1$,
$\inner{\vect{\beta}(u)}{\vect{\beta}(u)}=0$
imply that
$$
  0=PQ,\quad
  1=-\frac{E_v}{2} P + x Q.
$$
If $P(0)=0$, then $1=x(0) Q(0)$.
On the other hand,
by \eqref{eq:a-1st}, we have $x(0)=0$,
which is a contradiction.
Hence, we have $P(0)\ne0$.
Then, we have $Q=0$,
$P=-\frac2{E_v}$, and hence
$
  \vect{\beta}(u) = -\frac2{E_v}\,f_v.
$
The derivative $\vect{\beta}'(u)$
can be calculated as
$$
  \vect{b}'(u) \equiv
  \left( y-\frac{E_{uv}}{E_v} \right)\vect{b}(u).
$$
Since $\vect{b}'(u)\equiv \mu(u)\vect{b}(u)$ 
by \eqref{eq:Frenet-Lk},
we have \eqref{eq:tau-L-1}.
Moreover, since
$$
   \det(\vect{e},\,\vect{\kappa},\,\vect{\beta})
   = \det\left( f_u\,x f_v-\frac{E_v}{2}\psi,\,-\frac{2}{E_v}f_v \right)
   = -\det\left( f_u\,f_v,\,\psi \right)
   =-\ep,
$$
we have the desired result.
Since $\vect{\kappa}(0)$ and $\vect{\beta}(0)$
can be calculated as \eqref{eq:frame0-L},
we obtain \eqref{eq:frame0-L-1}.
\end{proof}

Let $p\in LD$ be a lightlike point of the second kind
of a mixed type surface $f:\Sigma\to \L^3$.
As in Definitions \ref{def:generic-surface} and \ref{def:generic-typeII},
$p$ is generic if 
there exists a non-null curve 
$ c(t)\,(|t|<\delta)$ passing through $p= c(0)$
such that the geodesic curvature function
along $ c(t)$ is unbounded at the origin.

\begin{proposition}\label{prop:STL-2}
Let $f: \Sigma \rightarrow \L^3$ be a mixed type surface,
and let $p\in \Sigma$ be a generic lightlike point of the second kind.
Let $ c(t)\,(|t|<\delta)$ be a non-null curve 
passing through $p= c(0)$
such that the geodesic curvature function
along $ c(t)$ is unbounded at the origin.
Let $(U;u,v)$ be an L-coordinate system 
associated with $ c(t)$ such that
the image of $ c(t)$ is included in 
the $u$-axis, and
$$
  ds^2 = E\,du^2 + G\,dv^2,\qquad
  E(u,0)=1
$$
hold.
Let $\psi$ be an L-Gauss map and set 
$$
  x(u):=X(u,0),\qquad 
  y(u):=Y(u,0),
$$
where
$X:=\inner{f_{uu}}{\psi}, Y:=\inner{f_{uv}}{\psi}$.
Then $\hat{c}(u):=f(u,0)$ is a unit-speed spacelike curve
of non-zero curvature.
Moreover, 
$\theta(u)$ is the causal curvature function of $\hat{c}(u)$
if and only if
\begin{equation}\label{eq:a-3}
    x(u) = -\frac{2\theta(u)}{E_v(u,0) + {\rm sgn}(E_v(0,0))
       \sqrt{E_v(u,0)^2 - 4\lambda(u,0) \theta(u)}}
\end{equation}
holds.
Furthermore, 
set $\calF:=(f_u,f_v,\psi)$
and $\ep\in \{1,-1\}$ as
$\ep=1$ $($resp.\ $\ep=-1)$
if the L-coordinate system $(u,v)$ is 
p-oriented $($resp.\ n-oriented$)$.
Then we have the following:
\begin{itemize}
\item
Consider the case that $\hat{c}(u)$ is a Frenet curve.
Then, $\tau(u)$ is the torsion function of $\hat{c}(u)$
if and only if
$y(u)$ is written as
\begin{equation}\label{eq:tau-ST-3}
  y(u)  = \ep\tau(u) 
  + \frac1{2\theta}
          \left( E_v x' - \lambda_u x^2 - E_{uv}x \right).
\end{equation}
Moreover, 
denote by
$\vect{e}(u)$ $($resp.\ $\vect{n}(u)$, $\vect{b}(u))$
the unit tangent vector field
$($resp.\ the principal normal vector field, 
the binormal vector field$)$ along $\hat{c}(u)$.
We also set $\sigma=-1$ $($resp.\ $\sigma=1)$
if $\hat{c}(u)$ is of type $S$ $($resp.\ type $T)$.
Then, $\calF(0,0)$ is given by \eqref{eq:frame0-ST-1}.
\item
Consider the case that $\hat{c}(u)$ is of type $L$.
Then, the L-coordinate system $(u,v)$ is p-oriented  
$($resp.\ n-oriented$)$ 
if and only if $\hat{c}(u)$ is of type $L^-$ $($resp.\ $L^+)$.
Moreover, $\mu(u)$ is the pseudo-torsion function of $\hat{c}(u)$
if and only if
$y(u)$ is written as
\begin{equation}\label{eq:tau-L-3}
  y(u) = \mu(u)  +\frac{E_{uv}}{E_v}.
\end{equation}
Furthermore, 
denote by
$\vect{e}(u)$ $($resp.\ $\vect{\kappa}(u)$, $\vect{\beta}(u))$
the unit tangent vector field
$($resp.\ the curvature vector field, the pseudo-binormal vector field$)$ 
along $\hat{c}(u)$.
Then, $\calF(0,0)$ is given by \eqref{eq:frame0-L-1}.
\item
Consider the case that $\hat{c}(u)$ is of type $L_k$ at $u=0$.
Then, the L-coordinate system $(u,v)$ is p-oriented 
$($resp.\ n-oriented$)$ 
if and only if $\hat{c}(u)$ is of type $L_k^-$ $($resp.\ $L_k^+)$.
Moreover, $\mu(u)$ is the pseudo-torsion function of $\hat{c}(u)$
if and only if
$y(u)$ is written as
\begin{equation}\label{eq:tau-Lk-3}
  y(u) = \mu(u)  
  +\frac{E_{uv} + \lambda x' + \lambda_u x}{E_v+\lambda x}.
\end{equation}
Furthermore, denote by
$\vect{e}(u)$ $($resp.\ $\vect{\kappa}(u)$, $\vect{\beta}(u))$
the unit tangent vector field
$($resp.\ the curvature vector field, the pseudo-binormal vector field$)$ 
along $\hat{c}(u)$.
Then, $\calF(0,0)$ is given by \eqref{eq:frame0-L-1}.
\end{itemize}
\end{proposition}

\begin{proof}
By Lemma \ref{lem:GW-matrix},
$$
  f_{uu} = x f_v - \left( \frac{E_v}{2} + \lambda x \right)\psi,\quad
  f_{uv} \equiv y f_v + \left( \frac{\lambda_u}{2} - \lambda y \right)\psi,\quad
  \psi_{u} \equiv -y \psi
$$
hold along the $u$-axis.
Here, $\vect{a} \equiv \vect{b}$
if $\vect{a}- \vect{b}$ is parallel to $\vect{e}(u)=f_u(u,0)$.
Moreover, we used the identity
$G(u,0)=\lambda(u,0)/E(u,0)=\lambda(u,0)$.
Since $\lambda(0,0)=0$,
$$
  \hat{c}''(0)
  =f_{uu}(0,0)
  =x(0) f_v(0,0) - \frac{E_v(0,0)}{2}\psi(0,0)
$$
holds.
By Lemma \ref{cor:limiting-geod-2nd},
the genericity of $ c$ at $p$ implies $E_v(0,0)=0$.
And hence, $\hat{c}''(0)\ne \vect{0}$ holds.

The causal curvature function 
$\theta(u):= \inner{\hat{c}''(u)}{\hat{c}''(u)}$
can be calculated as
$$
  \theta 
  = \inner{f_{uu}(u,0)}{f_{uu}(u,0)}
  = -E_v x - x^2 \lambda.
$$
Namely, 
\begin{equation}\label{eq:binary}
  \lambda(u,0) x(u)^2 + E_v(u,0) x(u) + \theta(u) =0
\end{equation}
holds.
Since the intersection of 
${\rm Image}( c)$ and $LD$ is $\{p\}$,
it holds that $\lambda(u,0)\neq0$ for $u\neq0$.
Then, $x(u)$ is given by
$$
  x(u)
  =\frac{-E_v(u,0)\pm\sqrt{E_v(u,0)^2-4\lambda(u,0) \theta(u)}
    }{2\lambda(u,0)}.
$$
Since the denominator converges to $0$
and $x(u)$ is bounded at $u=0$,
the limit of the numerator 
$$
  \lim_{u\to 0} \left(-E_v(u,0)\pm\sqrt{
  E_v(u,0)^2-4\lambda(u,0) \theta(u)} \right)
  = -E_v(0,0)\pm|E_v(0,0)|
$$
must be $0$.
Hence, we have 
\begin{align*}
  x(u) 
  &= \frac{-E_v(u,0) + {\rm sgn}(E_v)
       \sqrt{E_v(u,0)^2 - 4\lambda(u,0) \theta(u)}}{2\lambda(u,0)}\\
  &= -\frac{2\theta(u)}{E_v(u,0) + {\rm sgn}(E_v)
       \sqrt{E_v(u,0)^2 - 4\lambda(u,0) \theta(u)}},
\end{align*}
which yields \eqref{eq:a-3}.

\medskip
\noindent
{\it Case I {\rm :} $\hat{c}(u)$ is a Frenet curve.}
By \eqref{eq:GW-sub},
we have
$\hat{c}'(u)=f_u(u,0)$,
$\hat{c}''(u) = f_{uu}(u,0) = x f_v -(\frac{E_v}{2}+\lambda x)\psi$,
and 
$$
  \hat{c}'''(u)
  =f_{uuu}(u,0)
  \equiv (x' + xy) f_v 
    + \left(\frac1{2}E_v y
    -\frac{1}{2}E_{uv} - \lambda x' 
                 - \frac1{2} \lambda_u x \right)\psi.
$$
Applying \eqref{eq:binary}, we have
\begin{align*}
  \det(\hat{c}',\,\hat{c}'',\,\hat{c}''')
  &= \left( -\theta y
  +\frac1{2}\left( E_v x' -E_{uv} x - \lambda_u x^2 \right) \right) 
   \det\left( f_u,\, f_v, \, \psi \right)\\
  &= \ep \left( -\theta y
  +\frac1{2}\left( E_v x' -E_{uv} x - \lambda_u x^2 \right) \right),
\end{align*}
where 
$\ep=\det\left( f_u(u,0),\, f_v(u,0),\, \psi(u,0) \right) 
\in \{1,-1\}$
(cf.\ Lemma \ref{lem:detF}).
Hence, 
\eqref{eq:torsion-det} yields \eqref{eq:tau-ST-3}.

\medskip
\noindent
{\it Case II {\rm :} $\hat{c}(u)$ is of type $L$.}
Since $\theta(u)=0$, we have $x(u)=0$.
By \eqref{eq:GW-sub},
the curvature vector field
$\vect{\kappa}(u)=\hat{c}''(u)$
is written as
$\vect{\kappa}(u)=-\frac{E_v}{2}\psi(u,0)$.
On the other hand, 
since the pseudo-binormal vector field
$\vect{\beta}(u)$ is orthogonal to 
$\vect{e}(u)=f_u(u,0)$,
we may set
$\vect{\beta}(u)=P f_v +Q \psi$.
By
$\inner{\vect{\beta}(u)}{\vect{\kappa}(u)}=1$ and
$\inner{\vect{\beta}(u)}{\vect{\beta}(u)}=0$,
we have
$P=-2/E_v$, $Q=\lambda/E_v$.
Hence
$$
  \vect{\beta}(u)=\frac1{E_v}\left(-2f_v + \lambda \psi \right)
$$
holds.
Then,
substituting
$
  \vect{\kappa}'(u) 
  \equiv -\frac{1}{2}\left( E_{uv} -E_v y \right)\psi
$
into $\mu = -\inner{\vect{\beta}(u) }{\vect{\kappa}'(u) }$,
we have \eqref{eq:tau-L-3}.

\medskip
\noindent
{\it Case III {\rm :} $\hat{c}(u)$ is of type $L_k$.}
By \eqref{eq:GW-sub},
the curvature vector field
$\vect{\kappa}(u)=\hat{c}''(u)$
is written as
$\vect{\kappa}(u)
= x f_v - \left( \frac{E_v}{2} + \lambda x \right)\psi.$
Since the pseudo-binormal vector field
$\vect{\beta}(u)$
is orthogonal to $\hat{c}'(u)$,
we may set
$
  \vect{\beta}(u) = P(u)\,f_v + Q(u)\,\psi,
$
where $P(u)$, $Q(u)$ are functions.
If $P(0)=0$, then $\vect{\beta}(0)$ is 
parallel to $\psi(0,0)$.
On the other hand,
by \eqref{eq:a-3}, we have $x(0)=0$.
Hence, $\vect{\kappa}(0)$ is also parallel to $\psi(0,0)$,
which is a contradiction.
Therefore, $P(0)\ne0$.
Then, the conditions 
$\inner{\vect{\beta}(u)}{\vect{\kappa}(u)}=1$,
$\inner{\vect{\beta}(u)}{\vect{\beta}(u)}=0$
imply that
$
  P=-2/(E_v+\lambda x), Q=\lambda/(E_v+\lambda x),
$
namely,
$$
  \vect{b}(u) = \frac1{E_v + \lambda x}\left(-2f_v + \lambda\psi\right)
$$
holds.
The derivative $\vect{\beta}'(u)$
can be calculated as
$$
  \vect{b}'(u) \equiv 
  \left( y - \frac{E_{uv}+\lambda x' + \lambda_u x}{E_v+\lambda x} \right)\vect{b}(u).
$$
Since $\vect{b}'(u)\equiv \mu(u)\vect{b}(u)$ 
by \eqref{eq:Frenet-Lk},
we have \eqref{eq:tau-Lk-3}.

Finally, by a calculation similar to that 
done in the proof of Proposition \ref{prop:STL}, 
we obtain that
$\calF(0,0)$ is given by \eqref{eq:frame0-ST-1}
(resp.\ \eqref{eq:frame0-L-1})
if $\hat{c}(u)$ is of a Frenet curve (resp.\ a non-Frenet curve),
which yields the desired result.
\end{proof}

\section{Proof of main results}
\label{sec:proofs}
In this section,
we prove Theorem \ref{thm:main},
Corollaries \ref{cor:deformation} and \ref{cor:ext-kappa-N}
in the introduction.
We also prove the extrinsicity 
of the lightlike geodesic torsion $\kappa_G$
(Corollary \ref{cor:ext-kappa-G}).

\subsection{Isometric realizations of generic mixed type metrics}

Using Propositions \ref{prop:STL} and \ref{prop:STL-2},
we have the following.

\begin{theorem}\label{thm:realization}
Let $g$ be a real analytic generic mixed type metric
on a real analytic $2$-manifold $\Sigma$,
and let $p\in \Sigma$ be a semidefinite point.
We also let $c(t)$ $(|t|<\delta)$ be either
\begin{itemize}
\item a characteristic curve passing through $p=c(0)$, 
if $p$ is type I, or
\item a regular curve which is non-null at $p=c(0)$
such that the geodesic curvature function 
is unbounded at $t=0$, if $p$ is type II,
\end{itemize}
where $\delta>0$.
Take a real analytic spacelike curve $\gamma : I \to \L^3$
with non-zero curvature 
passing through $\gamma(0)=\vect{0}$
and set the image $\Gamma:=\gamma(I)$ of $\gamma$,
where $I$ is an open interval including the origin $0$.
Set $Z_\gamma$ as
$$
  Z_\gamma :=  
  \begin{cases} 
  \{1,2,3,4\} & (\text{if $\gamma$ is a Frenet curve}),\\ 
  \{1,2\} & (\text{if $\gamma$ is a non-Frenet curve}). 
  \end{cases}
$$
Then,
there exist a neighborhood $U$ of $p$
and real analytic mixed type surfaces 
$f_i : U \rightarrow \L^3$ $(i\in Z_\gamma)$
such that, for each $i\in Z_\gamma$,
\begin{itemize}
\item[$(1)$]
the first fundamental form of $f_i$ coincides with $g$,
\item[$(2)$]
$f_i(p)=\vect{0}$, and
the image of $f_i\circ c(t)$ is included in $\Gamma$.
\end{itemize}
Moreover, there are no such surfaces 
other than $f_i$ $(i\in Z_\gamma)$.
More precisely,
if $\tilde{f} : U \rightarrow \L^3$
is a real analytic generic mixed type surface 
which satisfies the conditions $(1)$ and $(2)$,
then there exists an open neighborhood $O$ of $p$ 
such that the image $\tilde{f}(O)$ is a subset 
of $f_i(U)$ for some $i\in Z_\gamma$.
\end{theorem}

\begin{proof}
Let $(V;u,v)$ be an L-coordinate system
along $ c(u)=(u,0)$.
Then, the metric $g$ can be represented as
$$
  g=E\,du^2+G\,dv^2,\qquad
  E(u,0)=1,\qquad
  E_v\neq0.
$$
Without loss of generality,
we may assume that 
the spacelike curve $\gamma(u)$
is parametrized by arclength
and $I$ is given by
$I=(-\delta,\delta)$, where $\delta$ is a positive real number.
Since $\gamma(u)$ has non-zero curvature
and is real analytic,
the spacelike curve $\gamma(u)$ is 
either a Frenet curve 
(i.e.\ type $S$ or type $T$)
or a non-Frenet curve
(i.e.\ type $L$ or type $L_k$ $(k=1,2,3,\dots)$).
Let $\theta(u)$ be the causal curvature function,
and let $\tau(u)$ (resp.\ $\mu(u)$) 
be the torsion (resp.\ pseudo-torsion) function of
$\gamma(u)$, respectively.


\medskip
\noindent
(i) {\it The case that $\gamma(u)$ is a Frenet curve}:~
Take $\ep\in \{1,-1\}$, arbitrarily.
\begin{itemize}
\item
If $p$ is type I,
we set functions $x(u)$, $y(u)$
as in \eqref{eq:a-1st}, \eqref{eq:tau-ST-1},
respectively.
\item
If $p$ is type II,
we set functions $x(u)$, $y(u)$
as in \eqref{eq:a-3}, \eqref{eq:tau-ST-3},
respectively.
\end{itemize}
Consider the system of partial differential equations
\begin{equation}\label{eq:normalform}
  X_v = \frac1{2E} \left( E_v X +E_u Y \right) + X\Delta - Y^2 +Y_u ,\quad
  Y_v = -\frac1{2E} \left( G_u X +E_v Y \right) +\Delta_u
\end{equation}
with unknown functions
$X(u,v)$, $Y(u,v)$,
where 
\begin{equation}\label{eq:C-Delta}
    \Delta := \frac1{2GX -E_v}\left(
        E_{vv}+G_{uu} - \frac{E_u G_u +E_v^2}{2E} 
        -G_v X+2 G_u Y + 2 GY^2
      \right).
\end{equation}
We remark that
$2GX -E_v\neq0$ holds 
on a neighborhood of the origin $(0,0)$,
since $G(0,0)=0$ and $E_v\neq0$.
By the Cauchy-Kowalevski theorem,
we have the unique solution 
$X(u,v)$, $Y(u,v)$, 
defined on a neighborhood $U$ of $(0,0)$,
to \eqref{eq:normalform}
with the initial condition
$$
  X(u,0)=x(u),\qquad 
  Y(u,0)=y(u).
$$
We set $Z(u,v):=\Delta$
and $h:=X\,du^2+2Y\,du\,dv+Z\,dv^2$.
Then, we have that
$X(u,v)$, $Y(u,v)$ and $Z(u,v)$ 
solve 
\eqref{eq:Cod1},
\eqref{eq:Cod2},
\eqref{eq:Gauss}.
Note that the Gauss equation \eqref{eq:Gauss} 
is equivalent to $Z=\Delta$.
By Corollary \ref{cor:fundamental}
with the initial condition
$$
  f(0,0)=\vect{0},\quad
  \calF(0,0)=\eqref{eq:frame0-ST-1},
$$
there exist an open neighborhood $U$ of $(0,0)$,
and a unique real analytic mixed type surface 
$f_1 : U\rightarrow \L^3$
(resp.\ $f_2 : U\rightarrow \L^3$)
having the L-Gauss map
$\psi_1 : U\rightarrow \Lambda^2$
(resp.\ $\psi_2 : U\rightarrow \L^3$)
such that 
\begin{itemize}
\item
the first fundamental form $ds^2$ of $f_1$ (resp.\ $f_2$) 
coincides with $g$,
\item
the second fundamental form 
$I\!I_{\psi_1}$
(resp.\ $I\!I_{\psi_2}$)
associated with $\psi_1$ 
(resp.\ $\psi_2$)
is given by $h$, and 
\item
$\det \calF_1>0$, namely, $\ep=1$
(resp.\ $\det\calF_2<0$, namely, $\ep=-1$),
\end{itemize}
where we set $\calF_i:=((f_i)_u, (f_i)_u, \psi_i)$
for $i=1,2$.
By Propositions \ref{prop:STL} and \ref{prop:STL-2},
we have that $\hat{c}_i(u):=f_i(u,0)$ 
is a unit-speed spacelike curve
whose causal curvature function
(resp.\ torsion function, pseudo-torsion function)
coincides with $\theta(u)$
(resp.\ $\tau(u)$, $\mu(u)$).
By Proposition \ref{prop:fund-ST},
we have $\hat{c}_i(u)=\gamma(u)$ for $i=1,2$.


\medskip
\noindent
(ii) {\it The case that $\gamma(u)$ is a non-Frenet curve}:~
Although the steps to follow are almost 
the same as in the case (i),
we do not have the ambiguity of 
$\ep\in \{1,-1\}$.
\begin{itemize}
\item
If $p$ is type I, 
we set functions $x(u)$, $y(u)$
as in \eqref{eq:a-1st}, \eqref{eq:tau-L-1},
respectively.
\item
If $p$ is type II,
we set $x(u)$ as in \eqref{eq:a-3}.
Also set $y(u)$ as in 
\eqref{eq:tau-L-3} (resp.\ \eqref{eq:tau-Lk-3})
if $\gamma(t)$ is of type $L$
(resp.\ type $L_k$ at $u=0$).
\end{itemize}
Consider the system of partial differential equations
\eqref{eq:normalform}
with unknown functions
$X(u,v)$, $Y(u,v)$,
where $\Delta$ is given by \eqref{eq:C-Delta}.
By the Cauchy-Kowalevski theorem,
we have the unique solution 
$X(u,v)$, $Y(u,v)$, 
defined on a neighborhood $U$ of $(0,0)$,
to \eqref{eq:normalform}
with the initial condition
$
  X(u,0)=x(u),~
  Y(u,0)=y(u).
$
We set $Z(u,v):=\Delta$
and $h:=X\,du^2+2Y\,du\,dv+Z\,dv^2$.
Then, we have that
$X(u,v)$, $Y(u,v)$ and $Z(u,v)$ 
solve 
\eqref{eq:Cod1},
\eqref{eq:Cod2},
\eqref{eq:Gauss}.
By Corollary \ref{cor:fundamental}
with the initial condition
$$
  f(0,0)=\vect{0},\quad
  \calF(0,0)=\eqref{eq:frame0-L-1},
$$
there exist an open neighborhood $U$ of $(0,0)$,
and a unique real analytic mixed type surface 
$f_1: U\rightarrow \L^3$
having the L-Gauss maps
$\psi_1 : U\rightarrow \Lambda^2$
such that 
\begin{itemize}
\item
the first fundamental form $ds^2$ of $f_1$ 
coincides with $g$,
\item
the second fundamental form $I\!I_{\psi_1}$
associated with $\psi_1$ is given by $h$.
\end{itemize}
By Propositions \ref{prop:STL} and \ref{prop:STL-2},
we have that $\hat{c}_1(u):=f_1(u,0)$ 
is a unit-speed spacelike curve
whose causal curvature function
(resp.\ torsion function, pseudo-torsion function)
coincides with $\theta(u)$
(resp.\ $\tau(u)$, $\mu(u)$).
By Propositions \ref{prop:fund-L} and \ref{prop:fund-Lk},
we have $\hat{c}_1(u)=\gamma(u)$.

On the other hand, 
$\bar{\gamma}(u):=\gamma(-u)$
also parametrizes $\Gamma$.
Hence, replacing $\gamma(u)$ with $\bar{\gamma}(u)$
in the above argument,
we obtain 
$f_3, f_4 : U\rightarrow \L^3$ 
(resp.\ $f_2 : U\rightarrow \L^3$)
if $\gamma$ is a Frenet curve
(resp.\ a non-Frenet curve),
which yields the desired result.
\end{proof}

As a direct corollary of Theorem \ref{thm:realization},
we obtain the following.

\begin{corollary}\label{cor:realization}
Let $g$ be a real analytic generic mixed type metric
on a connected real analytic $2$-manifold $\Sigma$.
For a semidefinite point $p\in S(g)$,
there exists a neighborhood $U$ of $p$
and a real analytic mixed type surface 
$f:U\to \L^3$
such that $g$ coincides 
with the first fundamental form $ds^2$ of $f$.
\end{corollary}

Corollary \ref{cor:realization} may be regarded 
as an analogue of the well-known Janet--Cartan theorem 
\cite{Janet, Cartan}.

\begin{proof}[Proof of Theorem \ref{thm:main}]
Set $g$ as the first fundamental form $g:=ds^2$ of $f$.
Applying Theorem \ref{thm:realization},
we obtain the desired surfaces $f_i$ $(i\in Z_\gamma)$.
\end{proof}

\begin{example}[The case of $\#Z_\gamma=4$]
\label{ex:four}
Let $f_1 : U \to \L^3$
be a mixed type surface defined by
$$
  f_1(u,v)=
  \begin{pmatrix}
  (1 - \alpha_+(u,v)) \cos u - 1\\
  (1 - \alpha_+(u,v)) \sin u\\
  \alpha_-(u,v)
  \end{pmatrix}
  \quad
  \left(
  \alpha_\pm(u,v):=(u+2)\left(v \pm \frac1{2} v^2\right)
  \right),
$$
where $U:=(-1,1)\times (-1/4,1/4)$.
Figure \ref{fig:f1-mixed}
shows the image of $f_1$.
The lightlike point set $LD$ of $f_1$
is given by the $u$-axis, and 
$\partial_v$ gives the null vector field,
which implies that 
every lightlike point $(u,0)\in LD$ of $f_1$
is of the first kind.
Moreover, every lightlike point $(u,0)$
is generic (i.e.\ the lightlike singular curvature 
$\kappa_L(u)$ does not vanish).
The lightlike set image $f_1(LD)$
is contained in the unit circle 
$$
  \Gamma:=\left\{(x_1,x_2,0)^T \in \L^3 
  \,;\, (x_1+1)^2+(x_2)^2=1 \right\}.
$$
We remark that 
$\Gamma$ is parametrized by
$\gamma(u)=(\cos u-1, \sin u,0)^T$,
which is a spacelike Frenet curve.
We set 
$$
f_2:=Sf_1,\quad 
f_3:=Rf_1,\quad 
f_4:=RSf_1,
$$
where $S$ (resp.\ $R$) is the matrix given in
\eqref{eq:S} (resp.\ \eqref{eq:R}).
Figure \ref{fig:four-sep} shows 
the images of these four surfaces $f_i$ for $i\in \{1,2,3,4\}$.
Since $S, R \in {\rm O}(1,2)$,
each $f_i$ $(i=2,3,4)$ is congruent to $f_1$.
And hence, 
the first fundamental forms of $f_i$ are the same.
Moreover, 
for each $i=1,2,3,4$,
the lightlike set images $f_i(LD)$  
is included in the unit circle $\Gamma$.
More precisely, it holds that
$$
  f_i(u,0)
  =\begin{cases}
    \gamma(u) & (\text{for $i=1,2$}),\\
    \gamma(-u) & (\text{for $i=3,4$}).
    \end{cases}
$$
Note that the orientation of 
the lightlike set images $f_i(LD)$ for $i=1,2$
does not coincide with 
that of $f_i(LD)$ for $i=3,4$.
So these surfaces satisfy the conditions $(1)$, $(2)$
in Theorem \ref{thm:main}. 
\begin{figure}[htb]
\begin{center}
 \begin{tabular}{{c@{\hspace{10mm}}c}}
  \resizebox{4cm}{!}{\includegraphics{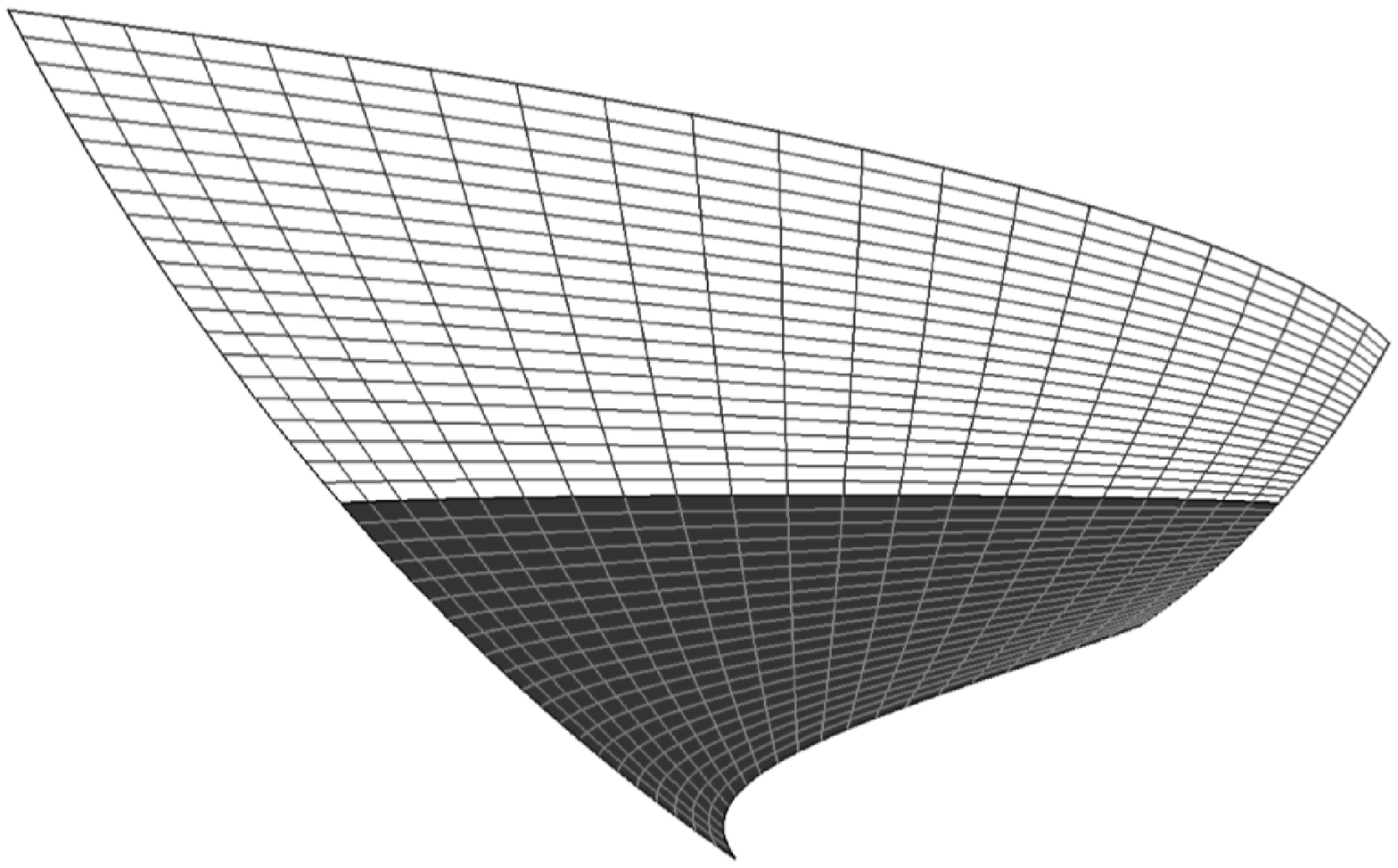}} &
  \resizebox{4cm}{!}{\includegraphics{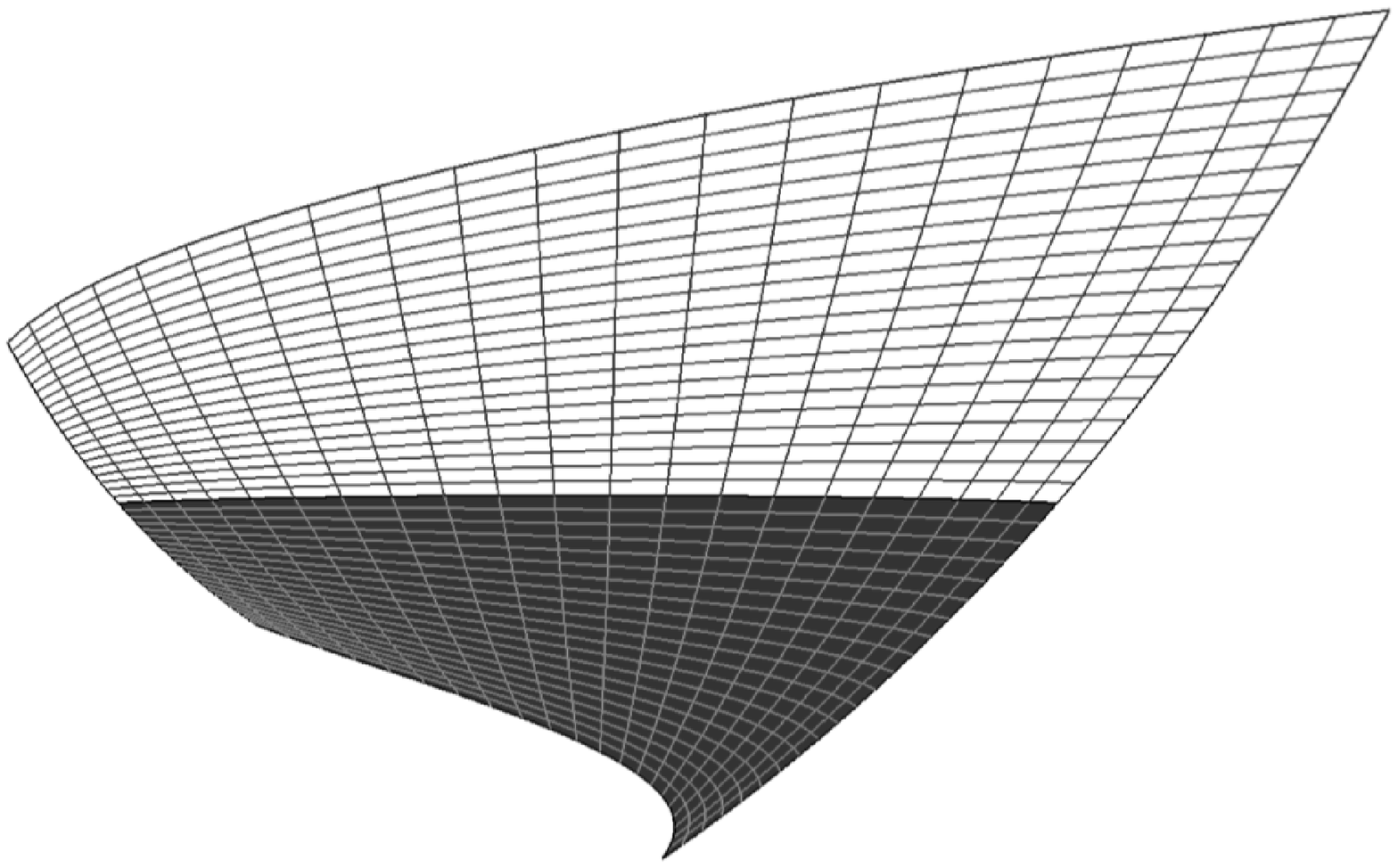}} \\
  {\footnotesize The image of $f_4$.} &
  {\footnotesize The image of $f_2$.} \\  
  \resizebox{4cm}{!}{\includegraphics{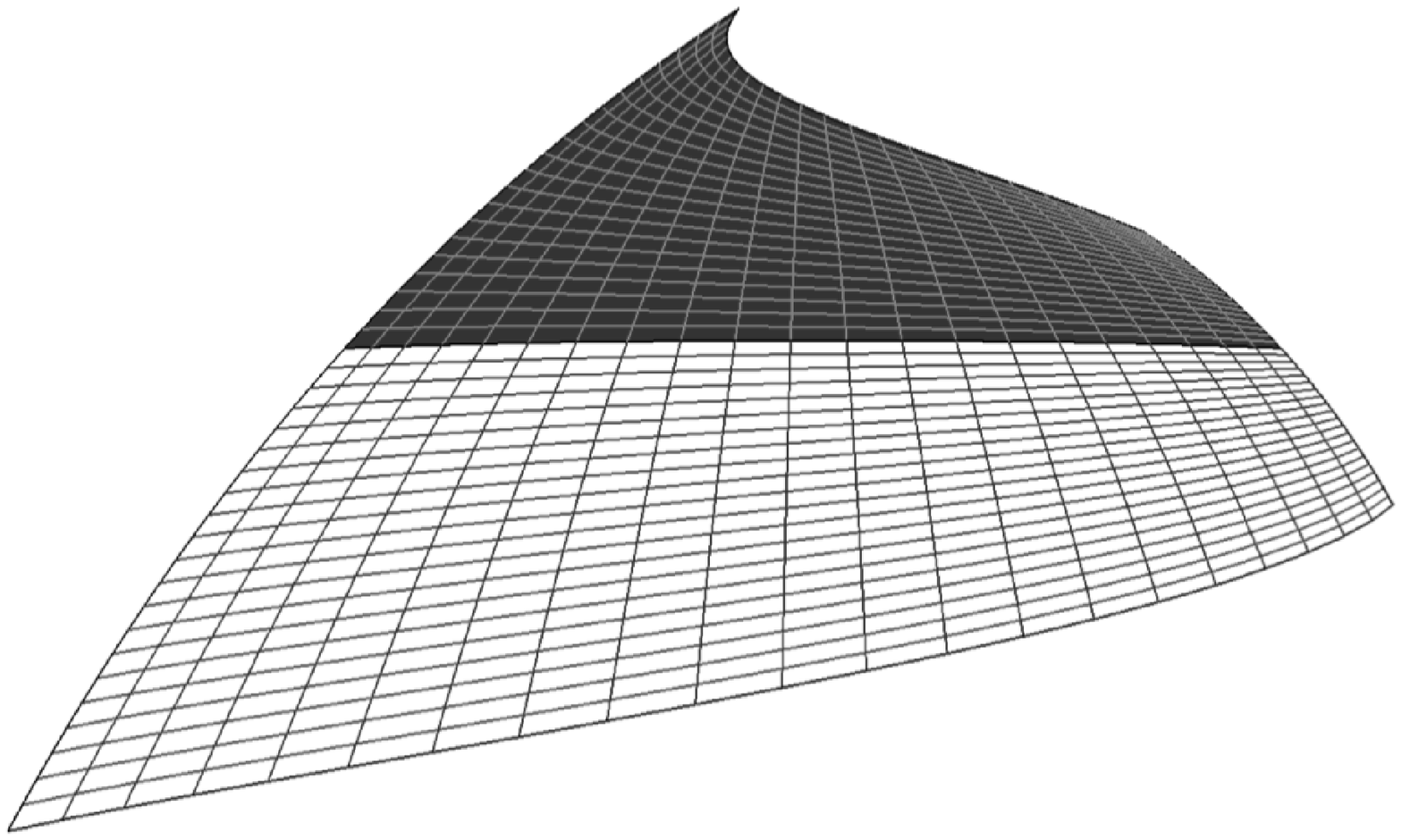}} &
  \resizebox{4cm}{!}{\includegraphics{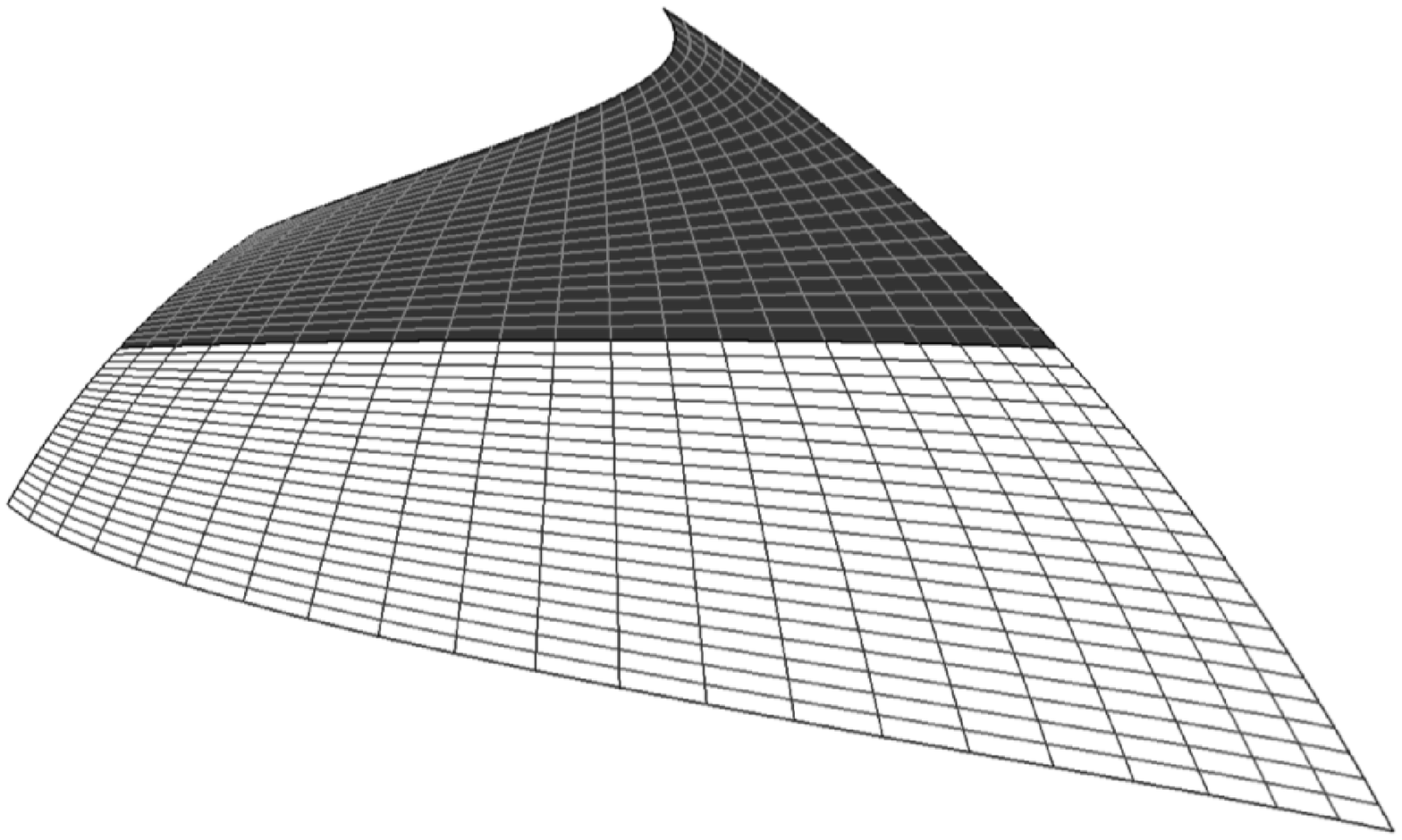}} \\
  {\footnotesize The image of $f_3$.} &
  {\footnotesize The image of $f_1$ (cf.\ Figure \ref{fig:f1-mixed}).} 
 \end{tabular}
 \caption{The images of the four mixed type surfaces $f_i$ 
 $(i=1,2,3,4)$ in Example \ref{ex:four}. 
 The dark (resp.\ light) colored region 
 is the image of spacelike (resp.\ timelike) point set.
 The boundary curve implies the lightlike set image,
 which is included in the unit circle $\Gamma$.
 Figure \ref{fig:four-sep} shows 
 the union $\cup_{i=1}^4 f_i(U)$ 
 of the images of $f_i$ $(i=1,2,3,4)$. 
 The image of $f_1$ is given in Figure \ref{fig:f1-mixed}.
 }
 \label{fig:four-sep}
\end{center}
\end{figure}
\end{example}

\begin{example}[The case of $\#Z_\gamma=2$]\label{ex:two}
Let $f_1 : U \to \L^3$
be a mixed type surface defined by
$$f_1(u,v)= \gamma(u)
 + (u+2) v\, \left(-\xi(u)+v\,\zeta\right),$$
where we set $U:=(-1,1)\times (-1/8,1/8)$,
and
$$  
  \gamma(u):=
  \begin{pmatrix} u\\ -u^2/2\\ u^2/2 \end{pmatrix},
  \quad
  \xi(u):=
  \begin{pmatrix} 2u\\ 1-u^2\\ 1+u^2 \end{pmatrix},
  \quad
  \zeta:=
  \begin{pmatrix} 0\\ -1\\ 1\end{pmatrix}.
$$
Figure \ref{fig:two-f1} shows the image of $f_1$.
Every lightlike point $(u,0)\in LD$ of $f_1$
is of the first kind.
In fact, the lightlike point set $LD$ of $f_1$
is given by the $u$-axis, and 
$\partial_v$ gives the null vector field.
Moreover, every lightlike point $(u,0)$
is generic (i.e.\ the lightlike singular curvature 
$\kappa_L(u)$ does not vanish).
The spacelike curve $\gamma(u)$
parametrizes
the lightlike set image $f_1(LD)$
of $f_1$.
Remark that $\gamma(u)$ is of type $L$,
and hence, it is a non-Frenet curve.
We set $f_2:=Mf_1$,
where $M$ is the diagonal matrix given by 
$M:={\rm diag}(-1,1,1)$.
Figure \ref{fig:two} shows 
the images of these two surfaces $f_1$ and $f_2$.
Since $M \in {\rm O}(1,2)$,
each $f_2$ is congruent to $f_1$.
And hence, 
the first fundamental forms of $f_1$ and $f_2$ are the same.
Moreover, 
the lightlike set image $f_1(LD)$ of $f_1$ 
coincides with $f_2(LD)$.
More precisely, it holds that
$$
  f_i(u,0)
  =\begin{cases}
    \gamma(u) & (\text{for $i=1$}),\\
    \gamma(-u) & (\text{for $i=2$}).
    \end{cases}
$$
Note that the orientation of 
the lightlike set images $f_1(LD)$
does not coincide with 
that of $f_2(LD)$.
So these surfaces satisfy the conditions $(1)$, $(2)$
in Theorem \ref{thm:main}. 
\begin{figure}[htb]
\begin{center}
 \begin{tabular}{{c@{\hspace{10mm}}c}}
  \resizebox{5cm}{!}{\includegraphics{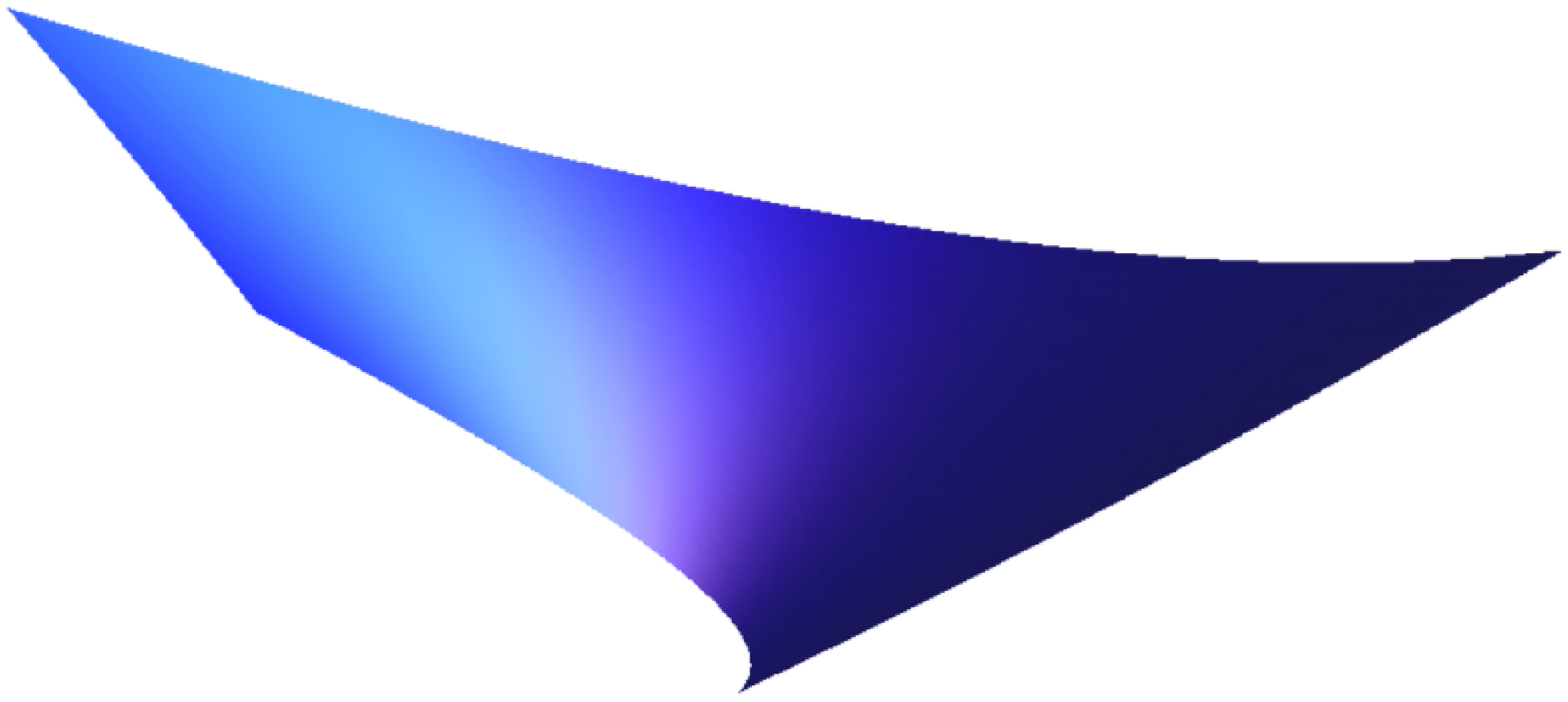}} &
  \resizebox{5cm}{!}{\includegraphics{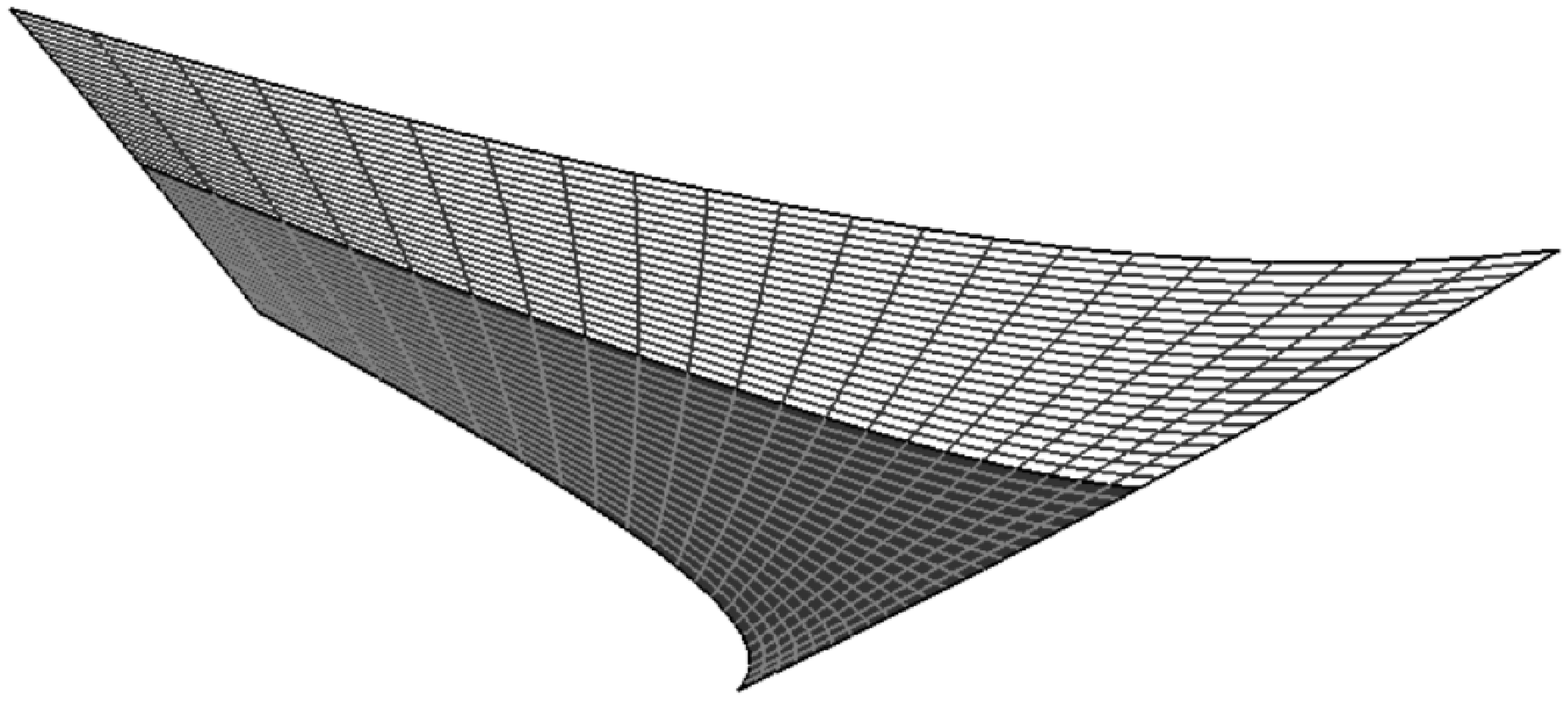}} 
 \end{tabular}
 \caption{Image of the mixed type surface $f_1$ 
 in Example \ref{ex:two}. 
 In the right figure,
 the dark (resp.\ light) colored region 
 is the image of spacelike (resp.\ timelike) point set.
 }
 \label{fig:two-f1}
\end{center}
\end{figure}
\begin{figure}[htb]
\begin{center}
 \begin{tabular}{{c@{\hspace{10mm}}c}}
  \resizebox{5cm}{!}{\includegraphics{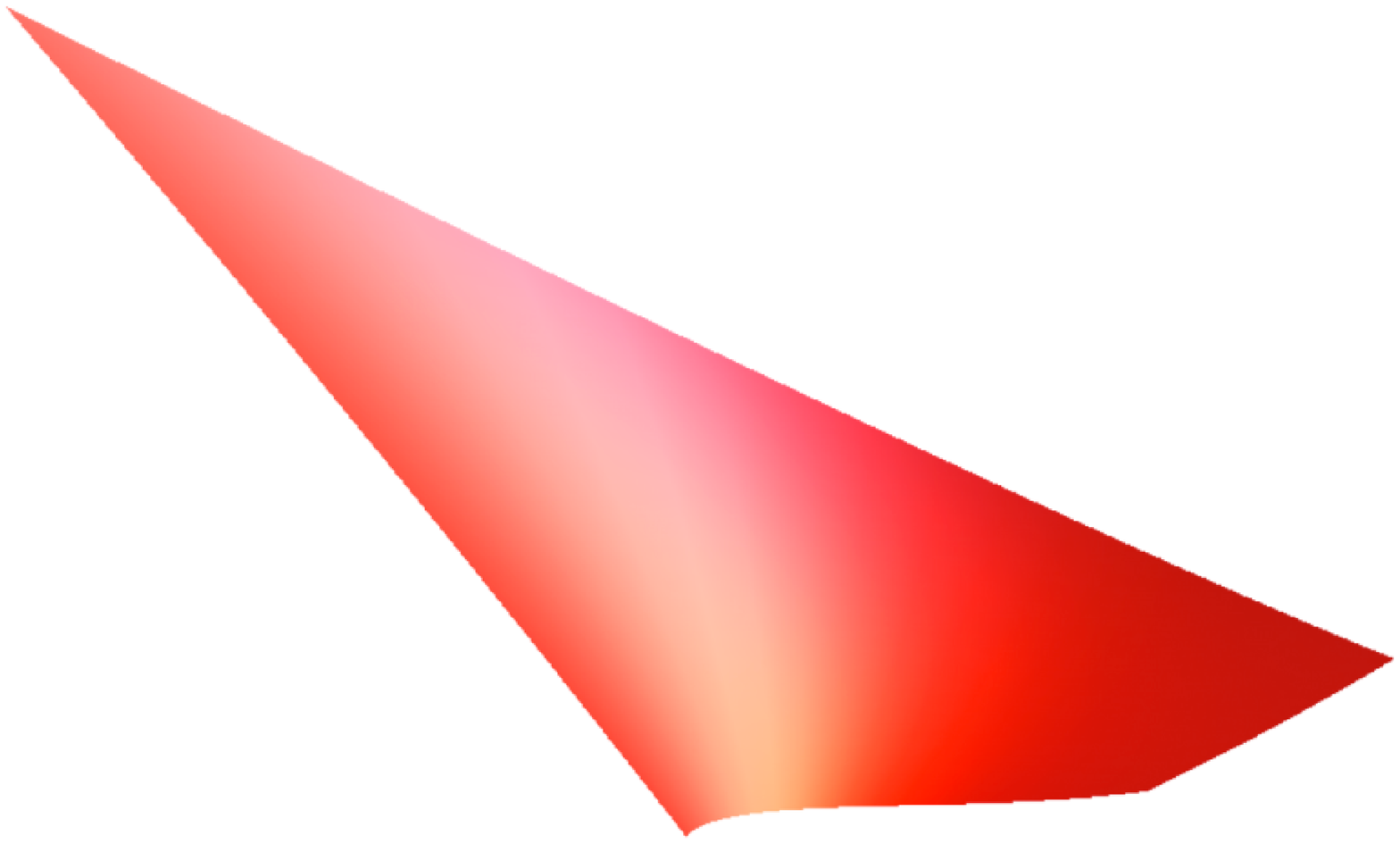}} &
  \resizebox{5cm}{!}{\includegraphics{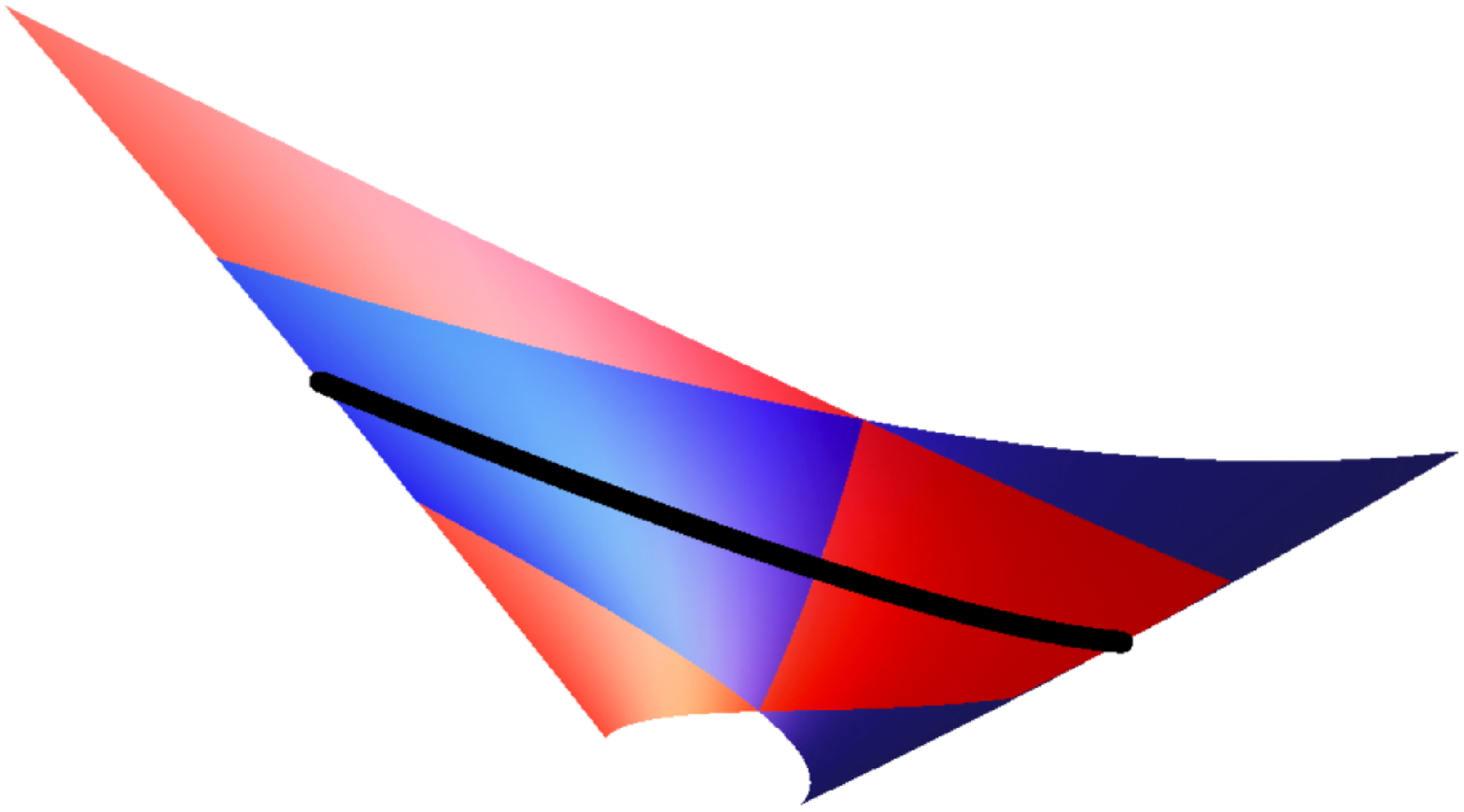}} 
 \end{tabular}
 \caption{The left figure shows the image
 of the mixed type surface $f_2$ 
 in Example \ref{ex:two}. 
 The right figure is the union of the images 
 of $f_1$ and $f_2$.
 The thick curve is the image of their lightlike point sets,
 which is a common spacelike non-Frenet curve.
 }
 \label{fig:two}
\end{center}
\end{figure}
\end{example}

\begin{proof}[Proof of Corollary \ref{cor:deformation}]
Let $f:\Sigma \to \L^3$ be a real analytic 
generic mixed type surface,
and let $p\in LD$ be a lightlike point.
Without loss of generality,
we may suppose that $f(p)=\vect{0}$.
Let $c(t)$ $(|t|<\delta)$ be either
\begin{itemize}
\item a characteristic curve passing through $p=c(0)$, 
if $p$ is of the first kind, or
\item a regular curve which is non-null at $p=c(0)$
such that the geodesic curvature function 
is unbounded at $t=0$, 
if $p$ is the second kind.
\end{itemize}
We set $\gamma(t):=f\circ c(t)$,
which is a real analytic spacelike curve
with non-zero curvature 
passing through $\gamma(0)=\vect{0}$.
We also set the image 
$\Gamma:=\gamma((-\delta,\delta))$ of $\gamma$.
Here, we deal with the case that 
$\gamma(t)=f\circ c(t)$ is a Frenet curve.
Similar proof can be applied 
to the case that 
$\gamma(t)$ is a non-Frenet curve.
Let $\theta(u)$ (resp.\ $\tau(u)$)
be the causal curvature 
(resp.\ the torsion) function 
of $\gamma(u)$.
We set 
$$
  \theta_s(u):= e^s \theta(u),\quad
  \tau_s(u):= \tau(u)
$$ 
for a constant $s\in [-1,1]$.
Let $\gamma_s(u)$ be a real analytic spacelike curve in $\L^3$
with non-zero curvature
such that 
the causal curvature (resp.\ the torsion) function of $\gamma_s(u)$ 
is given by $\theta_s(u)$ (resp.\ $\tau_s(u)$).
We remark that 
if $\gamma(u)$ is of type $S$ (resp.\ type $T$),
then so is $\gamma_s(u)$.
By a suitable choice of the initial condition of $\gamma_s(u)$,
we have that $\gamma_s(u)$ is 
a deformation of $\gamma(u)$,
namely, $\gamma_0(u)=\gamma(u)$ holds.
We set $\Gamma_s:=\gamma_s((-\delta,\delta))$ of $\gamma_s$.
By Theorem \ref{thm:realization}, 
for each $s$,
there exist a neighborhood $U_s$ of $p$
and four real analytic mixed type surfaces 
$f_i^s : U_s \rightarrow \L^3$ $(i=1,2,3,4)$
such that
\begin{itemize}
\item[$(1)$]
the first fundamental form of $f_i^s$ coincides with $ds^2$, 
\item[$(2)$]
$f_i^s(p)=\vect{0}$, and
the image of $f_i^s\circ c(t)$ is included in $\Gamma_s$
\end{itemize}
for each $i=1,2,3,4$.
By the uniqueness of Theorem \ref{thm:realization}, 
$f:U\to \L^3$ coincides with
either $f_1^0$, $f_2^0$, $f_3^0$ or $f_4^0$.
Without loss of generality,
we may assume that $f=f_1^0$ holds.
Then, $f^s:=f_1^s$ gives 
the desired non-trivial isometric deformation of $f$.
\end{proof}

\subsection{Extrinsicity of lightlike normal curvature}

In the case of vanishing 
lightlike singular curvature $\kappa_L=0$,
the lightlike normal curvature $\kappa_N$
is shown to be intrinsic:

\begin{fact}[{\cite[Corollary C]{HST}}]
\label{fact:N-int}
Let $f:\Sigma\to \L^3$ be a mixed type surface,
and let 
$p\in \Sigma$ be a lightlike point of the first kind.
If $\kappa_L=0$ holds along the characteristic curve near $p$,
then the lightlike normal curvature $\kappa_N$
is an intrinsic invariant.
\end{fact}

Hence, the remaining case is that 
the lightlike singular curvature 
$\kappa_L$ does not vanish at $p$,
namely $p$ is a generic lightlike point of the first kind.
Applying Theorem \ref{thm:main}, 
we prove the extrinsicity of $\kappa_N$
(Corollary \ref{cor:ext-kappa-N}).
We also prove the extrinsicity of $\kappa_G$ 
(Corollary \ref{cor:ext-kappa-G}).
To prove them, we prepare the following:

\begin{lemma}\label{lem:kappa-NG}
Let $f: \Sigma \rightarrow \L^3$ be a mixed type surface,
and let $p\in \Sigma$ be a generic lightlike point of the first kind.
Let $(U;u,v)$ be an L-coordinate system 
associated with the lightlike set at $p$.
Set $ds^2 = E\,du^2 + G\,dv^2$.
Then, 
the lightlike normal curvature $\kappa_N$ is written as
\begin{equation}\label{eq:kappa-n-uaxis}
  \kappa_N(u) = -\frac{\sqrt[3]{G_v(u,0)}}{E_v(u,0)} \,\theta(u).
\end{equation}
Moreover, set $\ep\in\{1,-1\}$ as
$\ep=1$ $($resp.\ $\ep=-1)$
if the L-coordinate system $(u,v)$ is 
p-oriented $($resp.\ n-oriented$)$.
Then,
the lightlike geodesic torsion $\kappa_G(u)$
along the $u$-axis is given by
\begin{equation}\label{eq:kappa-G-2FF}
  \kappa_G(u)
  = \begin{cases}
  \vspace{2mm}
      \dfrac{G_{uv}(u,0)}{3G_v(u,0)}
      - \ep\tau(u) -\dfrac{E_{uv}(u,0)}{E_v(u,0)} + \dfrac{\theta'(u)}{2\theta(u)}
  & (\text{if $\hat{c}(u)$ is Frenet}), \\
      \dfrac{G_{uv}(u,0)}{3G_v(u,0)} - \mu(u)  -\dfrac{E_{uv}(u,0)}{E_v(u,0)}
      & (\text{if $\hat{c}(u)$ is non-Frenet}),
    \end{cases}
\end{equation}
where $\hat{c}(u):=f(u,0)$, 
and $\theta(u)$ $($resp.\ $\tau(u)$, $\mu(u))$ is 
the curvature {\rm (}resp.\ torsion, pseudo-torsion{\rm )} function 
along $\hat{c}(u)$.
\end{lemma}

\begin{proof}
By \eqref{eq:L-normal-curvature},
$$
  \kappa_N(u)
  = \sqrt[3]{ \eta\inner{df(\eta)}{df(\eta)} }
     \inner{\hat{c}''(t)}{N(t)}
  = \sqrt[3]{G_v(u,0)} x(u)
$$
holds.
Then, \eqref{eq:a-1st} in Proposition \ref{prop:STL}
implies \eqref{eq:kappa-n-uaxis}.
By \eqref{eq:L-geodesic-torsion},
$\kappa_G(u)$ is written as
$$
  \kappa_G(u)
  = \inner{f_v}{\psi_u} + \frac{G_{uv}}{3G_v} 
  = -y(u) + \frac{G_{uv}}{3G_v} 
$$
along the $u$-axis.
By Proposition \ref{prop:STL},
we obtain \eqref{eq:kappa-G-2FF}.
\end{proof}

\begin{proof}[Proof of Corollary \ref{cor:ext-kappa-N}]
Let $f : \Sigma \to \L^3$ be a real analytic 
generic mixed type surface,
and let $p\in LD$ be a lightlike point of the first kind.
Taking an L-coordinate system $(V;u,v)$
associated with the characteristic curve 
passing through $p=(0,0)$,
the lightlike set image $f(LD)$
is parametrized by $\hat{c}(u):=f(u,0)$.
Let $\theta(u)$ be the causal curvature function of $\hat{c}(u)$
and set $$\theta_s(u):=\theta(u)+s$$ for some non-zero constant $s\in \R$.
Let $\gamma_s(u)$ be a real analytic spacelike curve in $\L^3$
with non-zero curvature
such that the causal curvature function of $\gamma(u)$
is given by $\theta_s(u)$.
By Theorem \ref{thm:main},
there exists a neighborhood $U\subset V$ of $p$
and a real analytic mixed type surface
$f^s : U \to \L^3$ such that 
$f^s$ and $f$ have the same fundamental form $ds^2$ on $U$,
and $f^s(u,0)=\gamma_s(u)$ $(|u|<\delta)$ 
for sufficiently small $\delta>0$.
Denote the lightlike normal curvature of $f$
(resp.\ $f^s$) along the $u$-axis
by $\kappa_N(u)$
(resp.\ $\kappa_N^s(u)$).
By Lemma \ref{lem:kappa-NG},
it holds that 
$$
  \kappa_N(u)-\kappa_N^s(u)
  = s \frac{\sqrt[3]{G_v(u,0)}}{E_v(u,0)}.
$$
Hence we have $\kappa_N^s(u)\ne \kappa_N(u)$,
which yields the desired result.
\end{proof}

Similarly, we obtain the following.

\begin{corollary}\label{cor:ext-kappa-G}
The lightlike geodesic torsion $\kappa_G$
is an extrinsic invariant.
\end{corollary}

\begin{proof}
As in the proof of Corollary \ref{cor:ext-kappa-N},
take a real analytic generic mixed type surface 
$f : \Sigma \to \L^3$
and a lightlike point $p\in LD$ of the first kind.
Suppose that the lightlike set image $f(LD)$ 
is a spacelike curve of type $S$ near $f(p)$.
Taking an L-coordinate system $(V;u,v)$
associated with the characteristic curve passing through $p=(0,0)$,
the parametrization $\hat{c}(u):=f(u,0)$ $(|u|<\delta)$ of $f(LD)$
satisfies $\theta(u)>0$ for sufficiently small $\delta>0$.
We also suppose that $\theta'(u)\ne0$.
Set $\theta_s(u):=\theta(u)+s$ for some positive constant $s\in \R$.
Let $\gamma_s(u)$ be a real analytic spacelike curve in $\L^3$
with non-zero curvature
such that the causal curvature function of $\gamma(u)$
is given by $\theta_s(u)$.

By Theorem \ref{thm:main},
there exists a neighborhood $U\subset V$ of $p$
and a real analytic mixed type surface
$f^s : U \to \L^3$ such that 
$f^s$ and $f$ have the same fundamental form $ds^2$ on $U$,
and $f^s(u,0)=\gamma_s(u)$ $(|u|<\delta)$ 
for sufficiently small $\delta>0$.
Denote 
the lightlike geodesic torsion of $f$
(resp.\ $f^s$) along the $u$-axis
by $\kappa_G(u)$
(resp.\ $\kappa_G^s(u)$).
By Lemma \ref{lem:kappa-NG},
it holds that 
$$
  \kappa_G(u)-\kappa_G^s(u)
  = \frac{s\theta'(u)}{2\theta(u)\theta_s(u)}.
$$
Hence we have $\kappa_G^s(u)\ne \kappa_G(u)$,
which yields the desired result.
\end{proof}

\begin{acknowledgements}
The author expresses gratitude to 
Wayne Rossman for careful reading of the first draft.
He also would like to thank Kentaro Saji, Keisuke Teramoto,
Masaaki Umehara and Kotaro Yamada for helpful comments.
\end{acknowledgements}


\end{document}